\documentclass{article}
\usepackage{amsfonts,latexsym,hyperref} 
\newcounter{satznum}
\newtheorem{theorem}{Theorem}[satznum]

\newcounter{examplenum}
\newtheorem{example}{Example}[examplenum]
\newcounter{lemmanum}
\newtheorem{lemma}{Lemma}[lemmanum]

\newenvironment{remark}
 {\begin{trivlist}\item[]{\bf Remark.}}
 {\end{trivlist}}
\newenvironment{remarks}
 {\begin{trivlist}\item[]{\bf Remarks.}}
 {\end{trivlist}}
\newenvironment{examples}
 {\begin{trivlist}\item[]{\bf Examples.}}
 {\end{trivlist}}
\newenvironment{proof}
 {\begin{trivlist}\item[]{\bf Proof.}}
 {\end{trivlist}}
{\makeatletter
\gdef\cz{{\mathbb C}} 
\gdef\me{{\mathbb E}} 
\gdef\nz{{\mathbb N}} 
\gdef\pr{{\mathbb P}} 
\gdef\rz{{\mathbb R}} 
}
\newcounter{todocounter}

\makeatletter
\def\@MRExtract#1 #2!{#1}
\newcommand{\MR}[1]{
  \xdef\@MRSTRIP{\@MRExtract#1 !}
  \href{http://www.ams.org/mathscinet-getitem?mr=\@MRSTRIP}{MR\@MRSTRIP}}
\makeatother

\begin{document}
   \section*{ASYMPTOTICS OF CONTINUOUS-TIME DISCRETE STATE SPACE
   BRANCHING PROCESSES FOR LARGE INITIAL STATE}
   {\sc Martin M\"ohle} and {\sc Benedict Vetter}\footnote{Mathematisches
   Institut, Eberhard Karls Universit\"at T\"ubingen, Auf der Morgenstelle 10,
   72076 T\"ubingen, Germany, E-mail addresses: martin.moehle@uni-tuebingen.de,
   benedict.vetter@uni-tuebingen.de}
\begin{center}
   \today
\end{center}
\begin{abstract}
   Scaling limits for continuous-time branching processes with discrete
   state space are provided as the initial state tends to infinity.
   Depending on the finiteness or non-finiteness of the mean and/or the
   variance of the offspring distribution, the limits are in general
   time-inhomogeneous Gaussian processes, time-inhomogeneous
   generalized Ornstein--Uhlenbeck type processes or continuous-state
   branching processes. We also provide transfer results showing how
   specific asymptotic relations for the probability generating
   function of the offspring distribution carry over to those of the
   one-dimensional distributions of the branching process.

   \vspace{2mm}

   \noindent Keywords: Branching process; generalized Mehler semigroup;
   Neveu's continuous-state branching process;
   Ornstein--Uhlenbeck type process; self-decomposability; stable law;
   time-inhomogeneous process; weak convergence

   \vspace{2mm}

   \noindent 2020 Mathematics Subject Classification:
            Primary
            60J80; 
            60F05  
            Secondary
            60F17; 
            60G50; 
            60J27  
\end{abstract}

\subsection{Introduction} \label{intro}
Suppose that the lifetime of each individual in some population is
exponentially distributed with a given parameter $a\in (0,\infty)$
and that at the end of its life each individual gives birth to
$k\in\nz_0:=\{0,1,\ldots\}$ individuals with probability $p_k$,
independently of the rest of the population. Assuming that the population
consists of $n\in\nz:=\{1,2,\ldots\}$ individuals at time $t=0$ we denote
with $Z_t^{(n)}$ the random number of individuals alive at time
$t\ge 0$. The process $Z^{(n)}:=(Z_t^{(n)})_{t\ge 0}$ is a
classical continuous-time branching process with discrete state space
$\nz_0\cup\{\infty\}$ and initial state $Z_0^{(n)}=n$. These processes
have been studied extensively in the literature. For fundamental
properties of these processes we refer the reader to the classical books
of Harris \cite[Chapter V]{harris} and Athreya and Ney
\cite[Chapter III]{athreyaney}. Define $Z_t:=Z_t^{(1)}$ and $Z:=Z^{(1)}$
for convenience. By the branching property, $Z^{(n)}$ is distributed as
the sum of $n$ independent copies of $Z$. The literature thus mainly
focuses on the situation $n=1$ and most results focus on the asymptotic
behavior of these processes as the time $t$ tends to infinity.

In contrast we are interested in the asymptotic behavior of $Z^{(n)}$ as
the initial state $n$ tends to infinity. To the best of the authors
knowledge this question has not been discussed rigorously in the
literature for continuous-time discrete state space branching processes.
Related questions for discrete-time Galton--Watson processes have been
studied extensively in the literature (see for example Lamperti
\cite{lamperti1,lamperti2} or Green \cite{green}), however in this
situation time is usually scaled as well, which make these
approaches different from the continuous-time case. The article of
Sagitov \cite{sagitov1} contains related results, however the
critical case is considered and again an additional time scaling is used.

The asymptotics as the initial state $n$ tends to infinity may in some sense
be viewed as a non-natural question in branching process theory, however this
question has fundamental applications, for example in coalescent theory. It is
well known that the block counting process of any exchangeable coalescent,
restricted to a sample of size $n$, has a Siegmund dual process, called
the fixation line. For the Bolthausen--Sznitman coalescent the fixation line
is (see, for example, \cite{kuklamoehle}) a continuous-time discrete state
space branching process $Z^{(n)}$ with offspring distribution $p_k=1/(k(k-1))$,
$k\in\{2,3\ldots\}$. In this context the parameter $n$ is the sample size
and hence the question about its asymptotic behavior when the sample size
$n$ size becomes large is natural and important. In fact, this example was
the starting point to become interested in the asymptotical behavior of
branching processes for large initial value.

The convergence results are provided in Section \ref{results}. We provide
a convergence result for the finite variance case (Theorem \ref{main1}),
another result for the situation when the process has still finite mean
but infinite variance (Theorem \ref{main2}) and for the situation
when even the mean is infinite but the process still does not explode
in finite time (Theorem \ref{main3}). The limiting processes arising
in Theorem \ref{main1} are (time-inhomogeneous) Gaussian processes whereas
those in Theorem \ref{main2} are (time-inhomogeneous) Ornstein--Uhlenbeck type
processes. In Theorem \ref{main3} continuous-state branching processes arise
in the limit as $n\to\infty$. For all three regimes typical examples
are provided. The basic idea to obtain convergence results
of this form is relatively obvious. Since $Z^{(n)}$ is a sum of $n$ independent
copies of $Z$ we can in principle apply central limit theorems, which
essentially lead to the desired results. We prove not only
convergence of the marginals or the finite-dimensional distributions.
We provide functional limiting results for the sequence of processes
$(Z^{(n)})_{n\in\nz}$. We think that the arising limiting processes are
quite interesting. For example, since the centering or scaling of the space
in Theorem \ref{main1} and Theorem \ref{main2} in general explicitly depends
on the time $t$, the limiting processes are in general time-inhomogeneous.

The convergence results are as well based on crucial transfer results
showing how particular asymptotic relations for the probability
generating function (pgf) of the offspring distribution carry over
to the pgf of $Z_t$. Results of this form are for example provided
in Lemma \ref{lemmaoffspringdist}, Lemma \ref{lemmaminfty} and
Lemma \ref{betatlemma} and are of its own interest. Despite the fact that
the literature on continuous-time branching processes is rather large, we
have not been able to trace these results in the literature.


Throughout the article $\xi$ denotes a random variable taking values in
$\nz_0$ with probability $p_k:=\pr(\xi=k)$, $k\in\nz_0$. For a space $E$
equipped with a $\sigma$-algebra we denote with $B(E)$ the space of all
bounded measurable functions $g:E\to\rz$. For a topological space $X$ and
$K\in\{\rz,\cz\}$ we denote by $\widehat{C}(X,K)$ the space of continuous
functions $g:X\to K$ vanishing at infinity and also write $\widehat{C}(X)$
for $\widehat{C}(X,\rz)$.
\subsection{Results} \label{results}
Let $f$ denote the pgf of $\xi$, i.e.
$f(s):=\me(s^\xi)=\sum_{k\ge 0}p_k s^k$ and define $u(s):=a(f(s)-s)$
for $s\in [0,1]$. Let $r\ge 1$. It is well known (see, for example,
Athreya and Ney \cite[p.~111, Corollary 1]{athreyaney}) that
$m_r(t):=\me(Z_t^r)<\infty$ for all $t>0$ if and only if $\me(\xi^r)
=\sum_{k\ge 0} k^rp_k<\infty$. Moreover $m(t):=m_1(t)=e^{\lambda t}$
with $\lambda:=u'(1-)=a(\me(\xi)-1)$ and
\begin{equation} \label{secondmoment}
   m_2(t)\ =\
   \left\{
   \begin{array}{cl}
      \tau^2\lambda^{-1}e^{\lambda t}(e^{\lambda t}-1)+e^{\lambda t}
                & \mbox{if $\lambda\ne 0$,}\\
      \tau^2t+1 & \mbox{if $\lambda=0$,}
   \end{array}
   \right.
\end{equation}
with $\tau^2:=u''(1-)=af''(1-)=a\me(\xi(\xi-1))$. Note that
(\ref{secondmoment}) slightly corrects Eq.~(5) on p.~109 in \cite{athreyaney},
which accidently provides the formula for the second descending factorial
moment $\me(Z_t(Z_t-1))$ instead of the second moment $\me(Z_t^2)$.
In particular, if $m_2(t)<\infty$, then
\[
\sigma^2(t)\ :=\ {\rm Var}(Z_t)\ =\
\left\{
   \begin{array}{cl}
      (\tau^2-\lambda)e^{\lambda t}(e^{\lambda t}-1)/\lambda & \mbox{if $\lambda\ne 0$,}\\
      \tau^2t & \mbox{if $\lambda=0$.}
   \end{array}
\right.
\]
\subsubsection{The finite variance case}
Assume that the second mean
$\me(\xi^2)=\sum_{k\ge 0} k^2p_k$ of the offspring distribution is finite or,
equivalently, that ${\rm Var}(Z_t)<\infty$ for all $t\ge 0$.
In the following $a\wedge b:=\min\{a,b\}$ denotes the minimum of
$a,b\in\rz$. We furthermore use for $\mu\in\rz$ and $\sigma^2\ge 0$
the notation $N(\mu,\sigma^2)$ for the normal distribution with mean
$\mu$ and variance $\sigma^2$ with the convention that $N(\mu,0)$
is the Dirac measure at $\mu$. Our first fluctuation result (Theorem \ref{main1})
clarifies the asymptotic behavior of $Z_t^{(n)}$ as the initial state $n$
tends to infinity. The proof of Theorem \ref{main1} is provided in
Section \ref{proofs}.
\begin{theorem} \label{main1}
   If $\me(\xi^2)<\infty$ or, equivalently, if $\sigma^2(t)
   :={\rm Var}(Z_t)<\infty$ for all $t\ge 0$, then, as $n\to\infty$,
   the process $X^{(n)}:=(X_t^{(n)})_{t\ge 0}$, defined via
   \begin{equation} \label{main1scaling}
      X_t^{(n)}\ :=\ \frac{Z_t^{(n)}-nm(t)}{\sqrt{n}}
      \ =\ \frac{Z_t^{(n)}-ne^{\lambda t}}{\sqrt{n}},
      \qquad n\in\nz,t\ge 0,
   \end{equation}
   converges in $D_\rz[0,\infty)$ to a continuous Gaussian Markov process
   $X=(X_t)_{t\ge 0}$ with $X_0=0$ and covariance function
   $(s,t)\mapsto {\rm Cov}(X_s,X_t)=\me(X_sX_t)=m(|s-t|)\sigma^2(s\wedge t)$,
   $s,t\ge 0$.
\end{theorem}
\begin{remarks}
\begin{enumerate}
   \item[1.] (Continuity of $X$) Let $s,t\ge 0$ and $x\in\rz$. Conditional
      on $X_s=x$ the random variable $X_{s+t}-X_s$ has a normal distribution
      with mean $\mu:=xm(t)-x=x(m(t)-1)$ and variance $v^2:=m(s)\sigma^2(t)$.
      Thus, $\me((X_{s+t}-X_s)^4\,|\,X_s=x)=3v^4+6\mu^2v^2+\mu^4
      =3m^2(s)\sigma^4(t)+6x^2(m(t)-1)^2m(s)\sigma^2(t)+x^4(m(t)-1)^4$ or,
      equivalently,
      \[
      \me((X_{s+t}-X_s)^4\,|\,X_s)\ =\
      3m^2(s)\sigma^4(t) + 6X_s^2(m(t)-1)^2m(s)\sigma^2(t)+X_s^4(m(t)-1)^4.
      \]
      Taking expectation yields
      \begin{eqnarray*}
         &   & \hspace{-15mm}\me((X_{s+t}-X_s)^4)\\
         & = & 3m^2(s)\sigma^4(t)
               + 6\me(X_s^2)(m(t)-1)^2m(s)\sigma^2(t)
               + \me(X_s^4)(m(t)-1)^4\\
         & = & 3m^2(s)\sigma^4(t)
               + 6\sigma^2(s)(m(t)-1)^2m(s)\sigma^2(t)
               + 3\sigma^4(s)(m(t)-1)^4.
      \end{eqnarray*}
      From this formula it follows that for every $T>0$ there exists a
      constant $K=K(T)\in (0,\infty)$ such that $\me((X_s-X_t)^4)\le K(s-t)^2$
      for all $s,t\in [0,T]$. By Kolmogorov's continuity theorem (see, for
      example, Kallenberg \cite[p.~57, Theorem 3.23]{kallenberg}) we can
      therefore assume that $X$ has continuous paths.
   \item[2.] (Generator) For $\lambda\ne 0$ the Gaussian process $X$ is
      time-inhomogeneous. Note that
      $T_{s,t}g(x):=\me(g(X_{s+t})\,|\,X_s=x)=\me(g(xm(t)+\sqrt{m(s)}X_t))$,
      $s,t\ge 0$, $g\in B(\rz)$, $x\in\rz$. Let $C^2(\rz)$ denote the space
      of real valued twice continuously differentiable functions on $\rz$.
      For $s\ge 0$, $g\in C^2(\rz)$ and $x\in\rz$ it follows that
      \[
      A_sg(x)
      \ :=\ \lim_{t\to 0} \frac{T_{s,t}g(x)-g(x)}{t}
      \ =\ \lambda xg'(x) + \frac{\sigma^2}{2}m(s)g''(x),
      \]
      where $\sigma^2:=\lim_{t\to 0}\sigma^2(t)/t=\tau^2-\lambda=a\me((\xi-1)^2)$.
      For $\lambda=0$ (critical case) the process $X$ is a time-homogeneous
      Brownian motion with generator $Ag(x)=(\tau^2/2)g''(x)$, $g\in C^2(\rz)$,
      $x\in\rz$, where $\tau^2=a{\rm Var}(\xi)$.
   \item[3.] (Doob--Meyer decomposition) Define $A:=(A_t)_{t\ge 0}$ via
      $A_t:=\lambda\int_0^t X_s\,{\rm d}s$, $t\ge 0$. Let
      ${\cal F}_t:=\sigma(X_s,s\le t)$, $t\ge 0$. For all $0\le s\le t$,
      \begin{eqnarray*}
      \me(A_t-A_s\,|\,{\cal F}_s)
      & = & \lambda\me\bigg(\int_s^t X_u\,{\rm d}u\,\bigg|\,{\cal F}_s\bigg)
      \ = \ \lambda\int_s^t\me(X_u\,|\,{\cal F}_s)\,{\rm d}u\\
      & = & \lambda\int_s^t m(u-s)X_s{\rm d}u
      \ = \ X_s\int_s^t \lambda e^{\lambda(u-s)}{\rm d}u\\
      & = & X_s(e^{\lambda(t-s)}-1)
      \ = \ X_sm(t-s)-X_s\\
      & = & \me(X_t\,|\,{\cal F}_s)-X_s
      \ = \ \me(X_t-X_s\,|\,{\cal F}_s).
      \end{eqnarray*}
      Thus, $M:=(M_t)_{t\ge 0}:=(X_t-A_t)_{t\ge 0}$ is a martingale with
      respect to the filtration $({\cal F}_t)_{t\ge 0}$. For $\lambda=0$
      the process $X$ itself is hence a martingale. Clearly, $X=M+A$ is
      the Doob--Meyer decomposition of $X$. The process $A$ is not monotone,
      but decomposes into $A=A^+-A^-$, where $A^+:=(A_t^+)_{t\ge 0}$ and
      $A^-:=(A_t^-)_{t\ge 0}$, defined via
      $A_t^+:=\lambda\int_0^t X_s^+\,{\rm d}s$ and
      $A_t^-:=\lambda\int_0^t X_s^-\,{\rm d}s$ for all $t\ge 0$, both have
      non-decreasing paths.
   \item[4.] (Positive semi-definiteness) The limiting process $X$ in
      Theorem \ref{main1} is Gaussian. For any finite number $k$ of time
      points $0\le t_1<\cdots<t_k<\infty$ it follows that
      $(X_{t_1},\ldots,X_{t_k})$ has a multivariate normal distribution
      with positive semi-definite covariance matrix
      $\Sigma:=(\sigma_{i,j})_{i,j\in\{1,\ldots,k\}}$ having entries
      $\sigma_{i,j}={\rm Cov}(X_{t_i},X_{t_j})=m(|t_i-t_j|)\sigma^2(t_i\wedge t_j)$,
      $i,j\in\{1,\ldots,k\}$. For $\lambda=0$ (critical case) it
      follows that the matrix $(t_i\wedge t_j)_{i,j\in\{1,\ldots,k\}}$
      is positive semi-definite. For further properties of such min and max
      matrices and related meet and join matrices we refer the reader
      exemplary to Bhatia \cite{bhatia1,bhatia2}
      and Mattila and Haukkanen \cite{mattilahaukkanen1,mattilahaukkanen2}.
      For $\lambda\ne 0$ (non-critical case)
      it follows that the matrix $(e^{\lambda|t_i-t_j|}
      e^{\lambda(t_i\wedge t_j)}(e^{\lambda(t_i\wedge t_j)}-1)/\lambda)_{i,j\in\{1,\ldots,k\}}$
      is positive semi-definite.
\end{enumerate}
\end{remarks}
\begin{examples}
   (i) Let $\xi$ is geometrically distributed with parameter $p\in (0,1)$. Define $q:=1-p$. Then
   all descending factorial moments $\me((\xi)_j)=j!(q/p)^j$, $j\in\nz_0$,
   are finite. Theorem \ref{main1} is hence applicable with
   $\lambda=a(\me(\xi)-1)=a(q/p-1)$ and $\tau^2=a\me((\xi)_2)=2a(q/p)^2$.
   For $p=1/2$ (critical case) the process $X$ is a Brownian motion
   with generator $Af(x)=af''(x)$, $f\in C^2(\rz)$, $x\in\rz$.

   (ii) If $\xi$ is Poisson distributed with parameter $\mu\in (0,\infty)$,
   then again all descending factorial moments $\me((\xi)_j)=\mu^j$,
   $j\in\nz_0$, are finite. Theorem \ref{main1} is applicable with
   $\lambda=a(\me(\xi)-1)=a(\mu-1)$ and $\tau^2=a\me((\xi)_2)=a\mu^2$.
   For $\mu=1$ (critical case) the process $X$ is a Brownian motion with
   generator $Af(x)=(a/2)\mu^2f''(x)$, $f\in C^2(\rz)$, $x\in\rz$.

   (iii) Let $a_1,a_2\ge 0$ with $a_1+a_2>0$. Theorem \ref{main1} is
   applicable for birth and death processes with rates $na_1$ and $na_2$
   for birth and death respectively if the process is in state $n$. In this
   case we have $a=a_1+a_2$, $f(s)=(a_2+a_1 s^2)/a$, $u(s)=a_2+a_1s^2-as$,
   $\lambda=a_1-a_2$ and $\tau^2=2a_1$. For $a_1=a_2$ (critical case)
   the process $X$ is a Brownian motion with generator $Af(x)=a_1f''(x)$,
   $f\in C^2(\rz)$, $x\in\rz$.
\end{examples}
\subsubsection{The finite mean infinite variance case}
In this subsection it is assumed that $m:=\me(\xi)<\infty$. Since $f$ is convex
on $[0,1]$ the inequality $1-f(s)\le m(1-s)$ holds for all $s\in [0,1]$.
In order to state appropriate limiting results it is usual to control
the difference between $m(1-s)$ and $1-f(s)$. A typical assumption of this
form is the following.

\vspace{2mm}

{\bf Assumption A.} There exists a constant $\alpha\in (1,2]$ and a function
$L:[1,\infty)\to (0,\infty)$ slowly varying (at infinity) such that
\begin{equation} \label{assumption}
   1-f(s)\ =\ m(1-s)-(1-s)^{\alpha}L((1-s)^{-1}),\qquad s\in [0,1).
\end{equation}
Since $f$ is differentiable, Assumption A in particular implies that $L$ is
differentiable. Define $F(s,t):=\me(s^{Z_t})$ for $s\in [0,1]$ and $t\ge 0$.
The following lemma clarifies the structure of $F(s,t)$ under Assumption A.
Recall that $m(t):=\me(Z_t)=e^{\lambda t}<\infty$.
\begin{lemma}
   If the offspring pgf $f$ satisfies Assumption A then, for every $t\ge 0$,
   \begin{equation} \label{consequence}
      1-F(s,t)
      \ =\ m(t)(1-s) - c(t)(1-s)^{\alpha}L((1-s)^{-1})(1+o(1)),
      \qquad s\to 1,
   \end{equation}
   where
   \begin{equation} \label{functionc}
      c(t)\ :=\
      \left\{
         \begin{array}{ll}
         at & \mbox{if $\lambda=0$ (critical case),}\\
         \displaystyle\frac{m(\alpha t)-m(t)}{(\alpha-1)(m-1)}
         \ =\ ae^{\lambda t}\frac{e^{\lambda(\alpha-1)t}-1}{(\alpha-1)\lambda}
         & \mbox{if $\lambda\ne 0$ (non-critical case).}
        \end{array}
      \right.
   \end{equation}
   \label{lemmaoffspringdist}
\end{lemma}
\begin{remark}
   Although we are in this subsection mainly interested in the infinite
   variance case, Lemma \ref{lemmaoffspringdist} holds in particular
   for the finite variance case. In this case Taylor expansion of $f$
   around $s=1$ shows that (\ref{assumption}) holds with $\alpha=2$ and
   $L((1-s)^{-1})\sim f''(1-)/2=\me(\xi(\xi-1))/2$ as $s\to 1$. Moreover,
   $c(t)f''(1-)=\me(Z_t(Z_t-1))=F''(1-,t)$, where $F''(s,t)$ denotes the
   second derivative of $F(s,t)$ with respect to $s$.
\end{remark}
In the following we are however interested in the infinite variance situation,
so we assume that $\me(\xi^2)=\infty$. We are now able to state our second
main convergence result.
\begin{theorem} \label{main2}
   Assume that $m:=\me(\xi)<\infty$ and $\me(\xi^2)=\infty$. Suppose that
   Assumption A holds, i.e. there exists a constant $\alpha\in (1,2]$ and
   a slowly varying function $L:[1,\infty)\to (0,\infty)$
   satisfying $\lim_{x\to\infty}L(x)=\infty$ and such that (\ref{assumption})
   holds. Let $(a_n)_{n\in\nz}$ be a sequence of positive real numbers
   satisfying $L(a_n)\sim a_n^\alpha/(\alpha n)$ as $n\to\infty$. Then
   the process $X^{(n)}:=(X_t^{(n)})_{t\ge 0}$, defined via
   \[
   X_t^{(n)}\ :=\ \frac{Z_t^{(n)}-nm(t)}{a_n},\qquad n\in\nz,t\ge 0,
   \]
   converges in $D_\rz[0,\infty)$ as $n\to\infty$ to a limiting process
   $X=(X_t)_{t\ge 0}$ with state space $\rz$ and initial state $X_0=0$,
   whose distribution is characterized as follows. Conditional on $X_s=x$
   the random variable $X_{s+t}$ is distributed as $xm(t)+(m(s))^{1/\alpha}X_t$,
   where $X_t$ is $\alpha$-stable with characteristic function
   $u\mapsto\me(e^{iuX_t})=\exp(c(t)(-iu)^\alpha/\alpha)$,
   $s,t\ge 0$, $u\in\rz$, and Laplace transform $\eta\mapsto
   \me(e^{-\eta X_t})=\exp(c(t)\eta^\alpha/\alpha)$,
   $\eta,t\ge 0$. Note that $\me(X_t)=0$, $t\ge 0$. The variance
   of $X_t$ is equal to $c(t)$ for $\alpha=2$ whereas ${\rm Var}(X_t)=\infty$
   for $t>0$ and $\alpha\in (1,2)$.
\end{theorem}
\begin{remark}
   As in Theorem \ref{main1} the limiting process $X$ in Theorem \ref{main2} is
   time-homogeneous if and only if $\lambda=0$. We have
   $T_{s,t}g(x):=\me(g(X_{s+t})\,|\,X_s=x)=\me(g(xm(t)+(m(s))^{1/\alpha}X_t))$
   for $s,t\ge 0$, $g\in B(\rz)$ and $x\in\rz$. Note that
   $T_{s,t}g(x)$ is well defined even for some functions $g$ which are not bounded.
   For example, for Laplace test functions of the
   form $g=g_\eta$, defined via $g_\eta(x):=e^{-\eta x}$ for all $x\in\rz$ and
   $\eta\ge 0$, we obtain the explicit formula
   \begin{eqnarray}
      A_sg_\eta(x)
      & := & \lim_{t\to 0}\frac{T_{s,t}g_\eta(x)-g_\eta(x)}{t}
      \ = \ \lim_{t\to 0}
            \frac{e^{-m(t)\eta x+c(t)m(s)\eta^\alpha/\alpha} - e^{-\eta x}}{t}
            \nonumber\\
      & = & \lim_{t\to 0} \Big(-m'(t)\eta x+c'(t)m(s)\frac{\eta^\alpha}{\alpha}\Big)
            e^{-m(t)\eta x+c(t)m(s)\eta^\alpha/\alpha}\nonumber\\
      & = & \Big(-m'(0+)\eta x + c'(0+)m(s)\frac{\eta^\alpha}{\alpha}\Big) e^{-\eta x}
            \nonumber\\
      & = & \Big(-\lambda\eta x + am(s)\frac{\eta^\alpha}{\alpha}\Big)e^{-\eta x},
            \qquad s,\eta\ge 0,x\in\rz. \label{explicit}
   \end{eqnarray}
   For $\alpha=2$ and $g\in C^2(\rz)$ it follows from (\ref{explicit}) that
   \[
   A_sg(x)\ :=\ \lim_{t\to 0}\frac{T_{s,t}g(x)-g(x)}{t}
   \ =\ \lambda xg'(x) + \frac{a}{2}m(s)g''(x),
   \qquad s\ge 0,x\in\rz,
   \]
   showing that for $\alpha=2$ the process $X$ has the same structure
   as in Theorem \ref{main1} with $\sigma^2$ replaced by the constant $a$.

   Assume now that $\alpha\in (1,2)$. Then, from (\ref{explicit}),
   a straightforward calculation based on the formula
   \[
   \int_0^\infty \frac{e^{-\eta h}-1+\eta h}{h^{\alpha+1}}\,{\rm d}h
   \ =\ \frac{\Gamma(2-\alpha)}{\alpha(\alpha-1)}\eta^\alpha\ =\ \Gamma(-\alpha)\eta^\alpha,
   \qquad \eta\ge 0, \alpha\in (1,2),
   \]
   yields
   \[
   A_sg(x)\ =\ \lambda xg'(x)+am(s)\frac{\alpha-1}{\Gamma(2-\alpha)}
   \int_0^\infty \frac{g(x+h)-g(x)-hg'(x)}{h^{\alpha+1}}\,{\rm d}h,
   \qquad s\ge 0, x\in\rz,
   \]
   first for $g=g_\eta$ and, hence, for
   other classes of functions $g$, for example for $g\in C_c^2(\rz)$.
   These formulas for the semigroup and the generator show that $X$ is
   a time-inhomogeneous Ornstein--Uhlenbeck type process
   \cite{shunxiangroeckner}. For fundamental results on such processes
   and related generalized Mehler semigroups we refer the reader to
   \cite{bogachevroecknerschmuland}.

   Even for $\alpha=2$ we have $a_n^\alpha/n\sim\alpha L(a_n)\to\infty$
   as $n\to\infty$, in contrast to the situation in Theorem \ref{main1},
   where $a_n=\sqrt{n}$ and, hence, $a_n^2/n=1$. For $\alpha=2$ the limiting
   random variable $X_t$ has a normal distribution with mean $0$ and variance
   $c(t)$ given via (\ref{functionc}) with $\alpha=2$.
\end{remark}
Two examples are now provided, one with $\alpha=2$ and the other with
$\alpha\in (1,2)$. In the first example the underlying branching process
is supercritical whereas in the second example it is critical. In the first
example $F(s,t)$ can be expressed in terms of the Lambert
$W$ function. In the second example $F(s,t)$ is known explicitly.
\begin{example}
   Suppose that $p_k=4/((k-1)k(k+1))$ for $k\in\{2,3,\ldots\}$, i.e.
   $f(s)=\sum_{k=2}^\infty p_ks^k=2s^{-1}(1-s)^2(-\log(1-s))-2+3s$,
   $s\in (0,1)$.
   Note that (\ref{assumption}) holds with $\alpha=2$,
   $m:=\me(\xi)=3$ and $L(x):=2(\log x)/(1-1/x)\sim2\log x$ as $x\to\infty$.
   Moreover, $\lambda=2a$,
   $m(t):=\me(Z_t)=e^{2at}$ and ${\rm Var}(Z_t)=\infty$ for $t>0$.
   The sequence $(a_n)_{n\in\nz}$,
   defined via $a_n:=\sqrt{2n\log n}$ for all $n\in\nz$, satisfies
   $L(a_n)\sim 2\log a_n\sim\log n=a_n^2/(2n)$ as $n\to\infty$.
   By Theorem \ref{main2}, the process
   $((Z_t^{(n)}-ne^{2at})/\sqrt{2n\log n})_{t\ge 0}$ converges
   in $D_\rz[0,\infty)$ as $n\to\infty$ to a time-inhomogeneous
   process $X=(X_t)_{t\ge 0}$ with distribution
   as described in Theorem \ref{main2}. In particular, for every $t>0$
   the random variable $X_t$ has a normal distribution with
   mean $0$ and variance $c(t)=\frac{1}{2}e^{2at}(e^{2at}-1)$.
   The pgf $F(.,t)$ of $Z_t$ can be computed as follows. From
   the backward equation
   \[
   t
   \ = \ \int_s^{F(s,t)} \frac{1}{u(x)}\,{\rm d}x
   \ = \ \frac{1}{a}\int_s^{F(s,t)}\frac{x}{2(1-x)((x-1)\log(1-x)-x)}\,{\rm d}x
   \ = \ \frac{1}{2a}[v(x)]_s^{F(s,t)}
   \]
   with $v(x):=\log(1-x)-\log(x+(1-x)\log(1-x))$, $x\in (0,1)$, we conclude
   that
   \begin{equation} \label{Fexplicit}
      F(s,t)\ =\ v^{-1}(2at+v(s)),
   \end{equation}
   where $v^{-1}:\rz\to (0,1)$ denotes the inverse of $v$, which turns out to
   be of the form $v^{-1}(y)=(1+W(h))/W(h)$, where $h:=-\exp(-1-e^{-y})
   \in (-1/e,0)$ and $W=W_{-1}$ denotes the lower branch of the Lambert $W$
   function satisfying $W(h)e^{W(h)}=h$ and being real valued on $[-1/e,0)$.
   Expansion of (\ref{Fexplicit}) shows that
   \[
   F(s,t)\ =\ 1 - e^{2at}(1-s)
   + e^{2at}(e^{2at}-1)(1-s)^2\log((1-s)^{-1}) + O((1-s)^2),
   \qquad s\to 1,
   \]
   in agreement with (\ref{consequence}), since
   $c(t)=\frac{1}{2}e^{2at}(e^{2at}-1)$ and $L(x)\sim 2\log x$ as $x\to\infty$.
\end{example}
\begin{example}
   Let $\alpha\in (1,2)$. Assume that $f(s)=s+(1-s)^\alpha/\alpha$,
   $s\in [0,1]$. Note that $p_0=1/\alpha$, $p_1=0$ and
   $p_k=(-1)^k{\alpha\choose k}/\alpha$ for $k\in\{2,3,\ldots\}$. In
   particular, $p_k\sim 1/(\alpha\Gamma(-\alpha) k^{\alpha+1})$ as
   $k\to\infty$. Moreover, $f'(s)=1-(1-s)^{\alpha-1}$ and, therefore,
   $m:=\me(\xi)=f'(1-)=1$. Thus, the underlying branching process is critical,
   the extinction probability is $q=1$ and (\ref{assumption}) holds
   with $L\equiv 1/\alpha$. Note that $u(s)=a(1-s)^\alpha/\alpha$.
   Theorem \ref{main2} is applicable with
   $a_n:=n^{1/\alpha}$. It follows that $(n^{-1/\alpha}(Z_t^{(n)}-n))_{t\ge 0}$
   converges in $D_\rz[0,\infty)$ as $n\to\infty$ to a process $X$ with distribution
   as described in Theorem \ref{main2}. In particular, for every $t\ge 0$ the
   random variable $X_t$ has characteristic function
   $u\mapsto \exp(-at(-iu)^\alpha/\alpha)$, $u\in\rz$. From
   \[
   at
   \ =\ \int_s^{F(s,t)}\frac{1}{f(x)-x}\,{\rm d}x
   \ =\ \int_s^{F(s,t)} \alpha(1-x)^{-\alpha}\,{\rm d}x
   \ =\ \frac{\alpha}{\alpha-1}
        \big((1-F(s,t))^{1-\alpha}-(1-s)^{1-\alpha}\big)
   \]
   it follows that $F(s,t)=1-
   ((\alpha-1)\alpha^{-1}ta + (1-s)^{1-\alpha})^{1/(1-\alpha)}$
   is known explicitly. Note that
   \begin{eqnarray*}
      1-F(s,t)
      & = & (1-s) - \frac{at}{\alpha}(1-s)^\alpha
            + \frac{a^2t^2}{2\alpha}(1-s)^{2\alpha-1} + O((1-s)^{3\alpha-2}),
            \qquad s\to 1,
   \end{eqnarray*}
   in agreement with (\ref{consequence}), since $c(t)=at$ and
   $L\equiv 1/\alpha$.
\end{example}
\subsubsection{The infinite mean case with non-explosion}
In this subsection it is assumed that $m:=\me(\xi)=\infty$ or, equivalently,
that $m(t):=\me(Z_t)=\infty$ for all $t>0$.
In order to state the result it is convenient to
define the function $L:[1,\infty)\to (0,\infty)$ via
\begin{equation} \label{Ldef}
   L(x)\ :=\ x(1-f(1-x^{-1})),\qquad x\ge 1.
\end{equation}
The substitution $s=1-x^{-1}$ shows that this definition is equivalent to
\begin{equation} \label{fLrelation}
   1-f(s)\ =\ (1-s)L((1-s)^{-1}),\qquad s\in [0,1).
\end{equation}
Non-explosion is assumed throughout this
section, which is equivalent to (see, for example, Harris
\cite[Chapter V, Section 9, p.~106, Theorem 9.1]{harris})
\[
\int_\varepsilon^1 \frac{1}{s-f(s)}\,{\rm d}s
\ =\ \int_{(1-\varepsilon)^{-1}}^\infty \frac{1}{x(L(x)-1)}\,{\rm d}x
\ =\ \infty
\]
for all $\varepsilon\in (q,1)$, where $q$ denotes the extinction
probability. For the theory of stable distributions and their domains
of attraction we refer the reader to Geluk and de Haan \cite{gelukdehaan}.
For the moment let $t>0$ be fixed. Then $Z_t^{(n)}$, suitably normalized,
converges in distribution as $n\to\infty$ to a non-degenerate limit,
that is, $Z_t$ is in the domain of attraction of a stable law, if and only
if the following condition is satisfied. There exists $\alpha(t)\in(0,1]$
and a slowly varying function $L_t:[1,\infty)\to(0,\infty)$
such that
\begin{equation}
   \pr(Z_t>x)\ \sim\ x^{-\alpha(t)}L_t(x),\qquad x\to\infty.
   \label{local_1}
\end{equation}
And, if $\alpha(t)=1$, then $L_t(x)\to\infty$ as $x\to\infty$.
In this subsection only the case $\alpha(t)<1$ is
investigated.
Recall that $F(s,t)=\me(s^{Z_t})$ for $s\in [0,1]$ and $t\ge 0$.
It follows from Bingham and Doney \cite{binghamdoney} that (\ref{local_1})
is then equivalent to
\begin{equation}
   1-F(s,t)\ =\ (1-s)^{\alpha(t)}L_t((1-s)^{-1}),\qquad s\in[0,1),
   \label{eqgeneratingfunc}
\end{equation}
where, to be precise, the function $L_t$ of (\ref{eqgeneratingfunc})
replaces $\Gamma(1-\alpha(t))L_t$.
%
%
Then,
\begin{equation} \label{alphatgeneral}
   \alpha(t)\ =\ \frac{\log\frac{1-F(s,t)}{L_t((1-s)^{-1})}}{\log(1-s)},
   \qquad t\ge 0,s\in [0,1).
\end{equation}
Since $L_t$ is slowly varying and hence satisfies
$\log L_t(x)/\log x\to 0$ as $x\to\infty$, it follows from
(\ref{alphatgeneral}) that
\begin{equation} \label{alphat}
\alpha(t)\ =\ \lim_{s\to 1}\frac{\log(1-F(s,t))}{\log(1-s)},
\qquad t\ge 0.
\end{equation}
In particular, $\alpha(t)$ is uniquely determined by the pgf $F(.,t)$.
Note that (\ref{eqgeneratingfunc}) always holds for $t=0$ with
$\alpha(0)=1$ and $c(0)=1$ because of the boundary condition $F(s,0)=s$.

Suppose (\ref{eqgeneratingfunc}) holds for all $t\ge 0$. From the
iteration formula $F(s,t+u)=F(F(s,t),u)$ it follows that
\begin{eqnarray*}
   &   & \hspace{-20mm}(1-s)^{\alpha(t+u)}L_{t+u}((1-s)^{-1})
   \ = \ 1-F(s,t+u)
   \ = \ 1-F(F(s,t),u)\\
   & = & (1-F(s,t))^{\alpha(u)}L_u((1-F(s,t))^{-1})\\
   & = & (1-s)^{\alpha(t)\alpha(u)} L_t^{\alpha(u)}((1-s)^{-1})
         L_u((1-s)^{-\alpha(t)}L_t^{-1}((1-s)^{-1})),
         \qquad s\in [0,1).
\end{eqnarray*}
Since all terms depending on $L_.$ are slowly varying, $\alpha(.)$
has to be multiplicative, i.e. $\alpha(t+u)=\alpha(t)\alpha(u)$
for all $t,u\ge 0$. The map $k:[0,\infty)\to [0,\infty)$, defined via
$k(t):=-\log\alpha(t)$ for all $t\ge 0$, is hence additive, so it
satisfies the Cauchy functional equation. By Aczel
\cite[p.~34, Theorem 1]{aczel}, $k(t)=Ct$ and, hence, $\alpha(t)=e^{-Ct}$
for all $t\ge 0$, where $C:=k(1)=-\log\alpha(1)\in [0,\infty)$. Clearly,
either $\alpha(t)=1$ for all $t\ge 0$, or $\alpha(t)<1$ for all $t>0$,
depending on whether $C=0$ or $C>0$. Also, the map $t\mapsto L_t(x)$
is continuously differentiable and satisfies the equation
\[
L_{t+u}((1-s)^{-1})\ =\
L_t^{\alpha(u)}((1-s)^{-1})L_u((1-s)^{-\alpha(t)}L_t^{-1}((1-s)^{-1})),
\qquad t,u\ge 0, s\in [0,1),
\]
or $L_{t+u}(x)=L_t^{\alpha(u)}(x)L_u(x^{\alpha(t)}L_t^{-1}(x))$ for all
$t,u\ge 0$ and all $x\ge 1$. The following result (Lemma \ref{lemmaminfty})
relates (\ref{eqgeneratingfunc}) to the pgf $f$ of the offspring
distribution of the branching process.
The map $s\mapsto L((1-s)^{-1})=\frac{1-f(s)}{1-s}$ has derivative
$s\mapsto\frac{1}{1-s}(\frac{1-f(s)}{1-s}-f'(s))$, which is strictly
positive on $[0,1)$ since $f$ is strictly convex. Thus, $L$ is strictly
increasing on $[1,\infty)$. We also have $L(x)\to\infty$ as $x\to\infty$
since $m=\infty$. The proof of Lemma \ref{lemmaminfty} is provided in
Section \ref{proofs3}.
%
%
%
\begin{lemma} \label{lemmaminfty}
   If $m:=f'(1-)=\infty$ then the following conditions are equivalent.
   \begin{enumerate}
      \item[(i)] For every $t>0$ there exists $\alpha(t)\in(0,1)$ and
         a slowly varying function $L_t:[1,\infty)\to (0,\infty)$ such
         that (\ref{eqgeneratingfunc}) holds.
      \item[(ii)] For every $t>0$ the limit
         \[
         \alpha(t)\ :=\ \lim_{s\to 1}\alpha(s,t)\ \in\ (0,1)
         \]
         exists, where $\alpha(s,t):=(1-s)(\frac{\partial}{\partial s}F(s,t))/(1-F(s,t))$
         for all $s\in [0,1)$.
      \item[(iii)] The limit
         \begin{equation} \label{Alimit}
         A\ :=\ \lim_{x\to\infty} \frac{L(x)}{\log x}
         \ =\ \lim_{s\to 1}\frac{1-f(s)}{(1-s)\log((1-s)^{-1})}
         \ \in\ (0,\infty)
         \end{equation}
         exists.
   \end{enumerate}
   In this case $\alpha(t)=e^{-aAt}$ for all $t\ge 0$.
\end{lemma}
\begin{remark}
   Note that
   \[
   aA
   \ =\ a\lim_{s\to 1}\frac{f(s)-1}{(1-s)\log(1-s)}
   \ =\ \lim_{s\to 1}\frac{u(s)-a(1-s)}{(1-s)\log(1-s)}
   \ =\ \lim_{s\to 1}\frac{u(s)}{(1-s)\log(1-s)}.
   \]
   Thus, $\alpha(t)=e^{-aAt}$ can be alternatively computed from the
   function $u(.)$.
\end{remark}
Suppose $m=\infty$ and that the limit $A:=\lim_{x\to\infty} L(x)/\log x
\in(0,\infty)$ in Lemma \ref{lemmaminfty} exists. Recall that, by Lemma
\ref{lemmaminfty}, the existence of the limit $A$ is equivalent to the
existence of constants $\alpha(t)\in (0,1)$ and of
slowly varying functions $L_t$ such that
(\ref{eqgeneratingfunc}) holds, i.e.
$1-F(s,t)=(1-s)^{\alpha(t)}L_t((1-s)^{-1})$. In the following we focus on
the particular situation that the limit
\begin{equation} \label{betatdef}
\beta(t)\ :=\
\lim_{x\to\infty} L_t(x)
\ =\ \lim_{s\to 1} L_t((1-s)^{-1})
\ =\ \lim_{s\to 1}\frac{1-F(s,t)}{(1-s)^{\alpha(t)}}\ \in\ (0,\infty)
\end{equation}
exists for each $t\ge 0$ and is neither $0$ nor $\infty$. We know already
that $\alpha(t)=e^{-aAt}$. If (\ref{betatdef}) holds, then we must have
$A>0$, since otherwise $\alpha(t)=1$ and hence $\beta(t)=m(t)=\infty$,
in contradiction to (\ref{betatdef}). The following result relates
(\ref{betatdef}) to the offspring's pgf $f$ and provides an explicit
formula for $\beta(t)$. The proof of Lemma \ref{betatlemma} is provided
in Section \ref{proofs3}.
\begin{lemma} \label{betatlemma}
   Suppose $m=\infty$ and that (\ref{Alimit})
   holds. If the limit $B:=\lim_{x\to\infty}(L(x)-A\log x)\in\rz$
   exists, then (\ref{betatdef}) holds for all $t\ge 0$. In this case
   \begin{equation} \label{betatformula}
      \beta(t)\ =\ \exp\bigg(\frac{B-1}{A}(1-\alpha(t))\bigg),\qquad t\ge 0.
   \end{equation}
\end{lemma}
We are now able to provide the third main convergence result.
In the following the notation $E:=[0,\infty)$ is used.
\begin{theorem} \label{main3}
   Suppose that $m=\infty$ and let $L$ be defined via (\ref{Ldef}) such
   that (see (\ref{fLrelation})) the relation $1-f(s)=(1-s)L((1-s)^{-1})$
   holds for all $s\in [0,1)$. Assume that both limits
   \[
   A\ :=\ \lim_{x\to\infty}\frac{L(x)}{\log x}\ \in\ (0,\infty)
   \quad\mbox{and}\quad
   B\ :=\ \lim_{x\to\infty}(L(x)-A\log x)\ \in\ \rz
   \]
   exist.
   For $t\ge 0$ define
   \begin{equation} \label{alphabeta}
      \alpha(t)\ :=\ e^{-aAt}\quad\mbox{and}\quad
      \beta(t)\ :=\ \exp\bigg(\frac{B-1}{A}(1-\alpha(t))\bigg).
   \end{equation}
   Then, as $n\to\infty$, the scaled process
   $X^{(n)}:=(X_t^{(n)})_{t\ge 0}$, defined via
   \[
   X_t^{(n)}\ :=\ n^{-1/\alpha(t)}Z_t^{(n)},\qquad t\ge 0,
   \]
   converges in $D_E[0,\infty)$ to a limiting continuous-state branching
   process $X=(X_t)_{t\ge 0}$, whose distribution is characterized as
   follows.
   \begin{enumerate}
      \item[i)] For every $t\ge 0$ the marginal random variable
         $X_t$ is $\alpha(t)$-stable with Laplace transform $\lambda\mapsto
         \exp(-\beta(t)\lambda^{\alpha(t)})$, $\lambda\ge 0$.
      \item[ii)] The semigroup $(T_t)_{t\ge 0}$ of $X$ satisfies
         $T_tg(x)=\me(g(x^{1/\alpha(t)}X_t))$, $x,t\ge 0$, $g\in B(E)$, i.e.
         conditional on $X_s=x$ the random variable $X_{s+t}$ has the same
         distribution as $x^{1/\alpha(t)}X_t$.
   \end{enumerate}
\end{theorem}
The proof of Theorem \ref{main3} is provided in Section \ref{proofs}.
We now provide three examples. In the first two examples the distribution of
$Z_t$ is known explicitly.
\begin{example} \label{neveu}
   Assume that $\xi$ has distribution $p_k:=\pr(\xi=k):=1/(k(k-1))$,
   $k\in\{2,3,\ldots\}$. Note that $\xi=\lfloor X\rfloor$, where $X$
   has density $f(x)=1/(x-1)^2$, $x\ge 2$, so $X$ has a shifted Pareto
   distribution with parameter $1$. Then,
   $f(s)=s+(1-s)\log(1-s)=1-(1-s)L((1-s)^{-1})$ with
   $L(x):=1+\log x$ and $u(s):=a(f(s)-s)=a(1-s)\log(1-s)$.
   Note that $A:=\lim_{x\to\infty}L(x)/\log x=1$ and
   $B:=\lim_{x\to\infty}(L(x)-\log x)=1$.
   From the backward equation
   $(\partial/\partial t)F(s,t)=u(F(s,t))$ it follows that
   \[
   t\ =\ \int_s^{F(s,t)}\frac{1}{u(x)}\,{\rm d}x
   \ =\ \frac{1}{a}\left[-\log(-\log(1-x))\right]_s^{F(s,t)}
   \ =\ \frac{1}{a}\log\Big(\frac{\log(1-s)}{\log(1-F(s,t))}\Big).
   \]
   Thus, $F(s,t)=1-(1-s)^{e^{-at}}$ showing that $Z_t$ is Sibuya distributed
   (see, for example, Christoph and Schreiber
   \cite[Eq.~(2)]{christophschreiber}) with parameter $e^{-at}$.
   The Sibuya distribution and similar distributions occur for example
   in Gnedin \cite[p.~84, Eq.~(9)]{gnedin},
   Huillet and M\"ohle \cite[p.~9]{huilletmoehle},
   Iksanov and M\"ohle \cite[p.~225]{iksanovmoehle} and
   Pitman \cite[p.~84, Eq.~(18)]{pitman1}, \cite[p.~70, Eq.(3.38)]{pitman2}.
   We conclude that (\ref{betatdef}) holds with $\alpha(t):=e^{-at}$ and
   $\beta(t):=1$. By Theorem \ref{main3}, as $n\to\infty$, the scaled process
   $X^{(n)}:=(Z_t^{(n)}/n^{e^{at}})_{t\ge 0}$ converges in $D_E[0,\infty)$
   to a limiting process $X=(X_t)_{t\ge 0}$ such that $X_t$ has Laplace
   transform $\lambda\mapsto\exp(-\lambda^{e^{-at}})$, $\lambda\ge 0$,
   and the semigroup $(T_t)_{t\ge 0}$ of $X$ satisfies
   $T_tg(x)=\me(g(x^{e^{at}}X_t))$, $x,t\ge 0$, $g\in B(E)$. We
   identify $(X_{t/a})_{t\ge 0}$ as Neveu's continuous-state branching
   process \cite{neveu}. For $a=1$ this example coincides with
   \cite[Theorem 2.1 b)]{kuklamoehle} stating that the fixation line of
   the Bolthausen--Sznitman $n$-coalescent, properly scaled, converges
   as $n\to\infty$ to Neveu's continuous-state branching process.
\end{example}
\begin{example} \label{neveugeneral}
   Example \ref{neveu} is easily generalized as follows. Fix two
   constants $b>0$ and $c\ge 0$ with $b+c\le 1$ and assume that
   $p_0:=c$, $p_1:=1-b-c$ and $p_k:=b/(k(k-1))$ for $k\ge 2$. Then
   $f(s)=s+(1-s)(c+b\log(1-s))=1-(1-s)(1-c-b\log(1-s))$,
   $u(s)=a(f(s)-s)=a(1-s)(c+b\log(1-s))$ and $L(x)
   =1-c+b\log x$. For $b=1$ and  $c=0$ we are back in
   Example \ref{neveu}. Note that $A:=\lim_{x\to\infty}L(x)/\log x=b>0$
   and $B:=\lim_{x\to\infty}(L(x)-b\log x)=1-c\in (0,1]$.
   The same argument as in Example \ref{neveu} leads to
   $F(s,t)=1-(1-s)^{e^{-abt}}\exp(cb^{-1}(e^{-abt}-1))$.
   Thus, Theorem \ref{main3} is applicable with $\alpha(t):=e^{-abt}$
   and $\beta(t):=\exp(cb^{-1}(e^{-abt}-1))$, $t\ge 0$. Clearly,
   these formulas for $\alpha(t)$ and $\beta(t)$ are in agreement with those
   from Lemma \ref{lemmaminfty} and Lemma \ref{betatlemma}, namely
   $\alpha(t)=e^{-aAt}=e^{-abt}$ and $\beta(t)=\exp((B-1)A^{-1}(1-\alpha(t)))
   =\exp(cb^{-1}(e^{-abt}-1))$, $t\ge 0$.
\end{example}
\begin{example} (Discrete Luria--Delbr\"uck distribution) \label{luria}
   Assume that $\xi$ has a discrete Luria--Delbr\"uck distribution with
   parameter $b\in (0,\infty)$, i.e. $f(s)=(1-s)^{b(1-s)/s}$, $s\in(0,1)$.
   Note that $f(0)=e^{-b}$ and $f(s)=1-(1-s)L((1-s)^{-1})$ for $s\in [0,1)$,
   where $L(1):=1-e^{-b}$ and $L(x):=x(1-x^{b/(1-x)})$ for $x\in(1,\infty)$.
   Note that $A:=\lim_{x\to\infty}L(x)/\log x=b$ and
   $B:=\lim_{x\to\infty}(L(x)-b\log x)=0$.
   Let $q=q(b)$ denote the extinction probability,
   i.e. the smallest fixed point of $f$ in the interval $[0,1]$. For all
   $\varepsilon\in (q,1)$,
   \[
   \int_\varepsilon^1 \frac{1}{s-f(s)}\,{\rm d}s
   \ =\ \int_{(1-\varepsilon)^{-1}}^\infty \frac{1}{x(L(x)-1)}\,{\rm d}x
   \ =\ \infty,
   \]
   since $L(x)\sim b\log x$ as $x\to\infty$. By the explosion criterion
   the associated branching process $Z=(Z_t)_{t\ge 0}$ does not explode.
   The functions $\alpha(.)$ and $\beta(.)$ are obtained as follows.
   By Lemma \ref{lemmaminfty}, $\alpha(t)=e^{-aAt}=e^{-abt}$, $t\ge 0$.
   Furthermore,
   \[
   \beta(t)\ =\ \exp\bigg(\frac{B-1}{A}(1-\alpha(t))\bigg)
   \ =\ \exp\bigg(\frac{e^{-abt}-1}{b}\bigg),\qquad t\ge 0.
   \]
   By Theorem \ref{main3}, as $n\to\infty$, the scaled process
   $X^{(n)}:=(Z_t^{(n)}/n^{e^{abt}})_{t\ge 0}$ converges in $D_E[0,\infty)$
   to a limiting process $X=(X_t)_{t\ge 0}$ such that $X_t$ has Laplace
   transform $\lambda\mapsto\exp(-\beta(t)\lambda^{e^{-abt}})$,
   $\lambda\ge 0$, and the semigroup $(T_t)_{t\ge 0}$ of $X$ satisfies
   $T_tg(x)=\me(g(x^{e^{abt}}X_t))$, $x,t\ge 0$, $g\in B(E)$.
\end{example}
The previous three examples are summarized in the following table.
\begin{center}
   \begin{tabular}{|c||c|c|c|}
   \hline
   Example & Example \ref{neveu} & Example \ref{neveugeneral} & Example \ref{luria}\\
   \hline\hline
   Parameters & --- & $b>0$, $c\ge 0$, $b+c\le 1$ & $0<b<\infty$\\
   \hline
   pgf $f(s)$ & $s+(1-s)\log(1-s)$ & $s+(1-s)(c+b\log(1-s))$ & $(1-s)^{b(1-s)/s}$\\
   \hline
   $L(x)$  & $1+\log x$ & $1-c+b\log x$ & $x(1-x)^{b/(1-x)}$\\
   \hline
   $\alpha(t)$ & $e^{-at}$ & $e^{-abt}$ & $e^{-abt}$\\
   \hline
   $\beta(t)$ & $1$ & $\exp(cb^{-1}(e^{-abt}-1))$ & $\exp((e^{-abt}-1)/b)$\\
   \hline
   \end{tabular}
\end{center}
\begin{remark}
   Theorem \ref{main3} does not cover the situation when the limit
   $A:=\lim_{x\to\infty}L(x)/\log x$ is either $0$ or $\infty$. We
   leave the analysis of the two boundary cases $A=0$ and $A=\infty$
   and of corresponding examples for future work.
\end{remark}
\subsubsection{The explosive case}
We briefly comment on the situation when the branching process may
explode in finite time. Note that explosion implies that
$A:=\lim_{x\to\infty}L(x)/\log x=\infty$. Thus, Theorem \ref{main3} is not
applicable. We have $F(1,t)<1$ for all $t>0$. For $t\ge 0$ let $G(.,t)$
denote the pgf of $Z_t$ conditioned on $Z_t<\infty$, i.e.
\[
G(s,t)\ :=\ \frac{F(s,t)}{F(1,t)},\qquad s\in [0,1], t\ge 0,
\]
In this situation a convergence result in the spirit of the previous
theorems, but with $F$ replaced by $G$, is obtained as follows.
For $t>0$ we have $\me(Z_t\,|\,Z_t<\infty)=G'(1-,t)=F'(1-,t)/F(1,t)=\infty$.
Thus, it is natural to assume that
$1-G(s,t)=(1-s)^{\alpha(t)}L_t((1-s)^{-1})$ for some $\alpha(t)\in (0,1]$
and some slowly varying function $L_t$. Assume now
furthermore that the limits
\[
\beta(t)\ :=\ \lim_{x\to\infty}L_t(x)\ \in\ (0,\infty),\qquad t\ge 0,
\]
exist. Then $\alpha(t)<1$ for all $t>0$. Now, for $t\ge 0$ and $n\in\nz$ choose
$a_n(t)$ such that $L_t(a_n(t))\sim (a_n(t))^{\alpha(t)}/(n\alpha(t))$ as
$n\to\infty$. Then $Z_t^{(n)}/a_n(t)$, conditioned on $Z_t<\infty$, converges
to $X_t$ in distribution as $n\to\infty$, where $X_t$ has Laplace transform
$\lambda\mapsto\exp(-\beta(t)\lambda^{\alpha(t)})$, $\lambda\ge 0$.
Example \ref{sibuya} below turns out to be in that regime.
\begin{example} \label{sibuya}
   Suppose that $\xi$ is Sibuya distributed with parameter
   $\alpha\in (0,1)$, i.e. $f(s)=1-(1-s)^\alpha$, $s\in [0,1]$.
   Note that $f(s)=1-(1-s)L((1-s)^{-1})$, where
   $L(x):=x^{1-\alpha}$ is regularly varying of index $1-\alpha$.
   From the backward equation
   \begin{eqnarray*}
      at
      & = & \int_s^{F(s,t)} \frac{1}{f(x)-x}\,{\rm d}x
      \ = \ \int_s^{F(s,t)} \frac{1}{1-x-(1-x)^\alpha}\,{\rm d}x\\
      & = & \bigg[\frac{-\log(1-(1-x)^{1-\alpha})}{1-\alpha}\bigg]_s^{F(s,t)}
      \ = \ \frac{1}{1-\alpha}\log\frac{1-(1-s)^{1-\alpha}}{1-(1-F(s,t))^{1-\alpha}},
      \qquad t\ge 0,
   \end{eqnarray*}
   we obtain the explicit solution
   \begin{equation} \label{sibuyaexplicit}
      F(s,t)\ =\ 1 - \Big(
         1- e^{-(1-\alpha)at} (1-(1-s)^{1-\alpha})
      \Big)^{\frac{1}{1-\alpha}},
      \qquad s\in [0,1], t\ge 0.
   \end{equation}
   We have $\pr(Z_t=\infty)=1-F(1,t)=(1-e^{-(1-\alpha)at})^{\frac{1}{1-\alpha}}$
   for $t\ge0$, so $0<\pr(Z_t=\infty)<1$ for all $t>0$.
   The time
   $T:=\inf\{t>0\,:\,Z_t=\infty\}$ of explosion satisfies
   $\pr(T<\infty)=\lim_{t\to\infty}\pr(Z_t=\infty)=1$, so $Z$ explodes
   in finite time almost surely. Note that $T$ has mean
   \[
   \me(T)
   \ =\ \int_0^\infty \pr(T>t)\,{\rm d}t
   \ =\ \int_0^\infty \pr(Z_t<\infty)\,{\rm d}t
   \ =\ \int_0^\infty (1-(1-e^{-(1-\alpha)at})^{\frac{1}{1-\alpha}})\,{\rm d}t.
   \]
   The substitution $x=1-e^{-(1-\alpha)at}$ yields
   \[
   \me(T)\ =\ \frac{1}{a(1-\alpha)}\int_0^1\frac{1-x^{\frac{1}{1-\alpha}}}{1-x}\,{\rm d}x
   \ =\ \frac{1}{a(1-\alpha)}\bigg(\Psi\bigg(\frac{2-\alpha}{1-\alpha}\bigg)+\gamma\bigg),
   \]
   where $\Psi=\Gamma'/\Gamma$ denotes the logarithmic derivative of the
   gamma function and $\gamma$ is the Euler--Mascheroni constant.

   Let $t>0$ in the following. Expansion of (\ref{sibuyaexplicit})
   yields 
   \begin{equation} \label{expansion}
      F(s,t)\ =\ F(1,t)
      - \frac{1}{1-\alpha}
      (1-e^{-(1-\alpha)at})^{\frac{\alpha}{1-\alpha}}e^{-(1-\alpha)at}
      (1-s)^{1-\alpha}
      + O((1-s)^{2(1-\alpha)}),\quad s\to 1.
   \end{equation}
   Rewriting (\ref{expansion}) in the form
   \begin{eqnarray*}
      1-G(s,t)
      & = & 1-\frac{F(s,t)}{F(1,t)}\\
      & = & \frac{(1-e^{-(1-\alpha)at})^{\frac{\alpha}{1-\alpha}}e^{-(1-\alpha)at}}
   {(1-\alpha)(1-(1-e^{-(1-\alpha)at})^{\frac{1}{1-\alpha}})}
   (1-s)^{1-\alpha} + O((1-s)^{2(1-\alpha)}),
   \qquad s\to 1,
   \end{eqnarray*}
   yields $\alpha(t)=1-\alpha$ for all $t>0$ and
   \[
   \beta(t)\ :=\ \lim_{x\to\infty}L_t(x)
   \ =\ \frac{(1-e^{-(1-\alpha)at})^{\frac{\alpha}{1-\alpha}}e^{-(1-\alpha)at}}
  {(1-\alpha)(1-(1-e^{-(1-\alpha)at})^{\frac{1}{1-\alpha}})},\qquad t>0.
   \]
   Thus, the sequence $a_n(t):=(n\alpha(t)\beta(t))^{1/\alpha(t)}$ satisfies
   $L_t(a_n(t))\sim (a_n(t))^{\alpha(t)}/(n\alpha(t))$ as $n\to\infty$ and
   it follows that $X_t^{(n)}:=Z_t^{(n)}/a_n(t)$, conditioned on
   $Z_t<\infty$, converges to $X_t$ in distribution as $n\to\infty$,
   where $X_t$ has Laplace transform
   $\lambda\mapsto\exp(-\beta(t)\lambda^{\alpha(t)})$, $\lambda\ge 0$.
\end{example}
We leave the study of further examples of branching processes with
explosion similar to those of Example \ref{sibuya} to the interested
reader. One may for instance study the pgf $f(s):=\frac{2}{\pi}\arcsin s$,
$s\in [0,1]$, occurring in Pakes \cite[p.~276, Example 4.5]{pakes}.
A further example is the offspring distribution
$p_k=\frac{\sqrt{\pi}}{4}\Gamma(k)/\Gamma(k+3/2)$, $k\in\nz$,
in which case the offspring pgf has the form
$f(s)=1-\sqrt{(1-s)/s}\arcsin\sqrt{s}$.

Let us finally discuss the situation when
\begin{equation} \label{gcond}
   1-G(s,t)\ =\ (1-s)L_t((1-s)^{-1}),\qquad t\ge 0,
\end{equation}
for some slowly varying function $L_t$. Note that (see, for example,
Bingham and Doney \cite[Theorem A]{binghamdoney}) (\ref{gcond}) is
equivalent to $\sum_{k=0}^n\pr(Z_t>k\,|\,Z_t<\infty)\sim L_t(n)$ as
$n\to\infty$, which is Condition (ii) in Rogozin's relative stability
theorem (see, for example, Bingham, Goldie and Teugels
\cite[Theorem 8.8.1]{binghamgoldieteugels}). Let $(a_n(t))_{n\in\nz}$
be a sequence such that $L_t(a_n(t))\sim a_n(t)/n$ as $n\to\infty$.
Then, by Theorem 8.8.1 of \cite{binghamgoldieteugels},
$Z_t^{(n)}/a_n(t)|_{Z_t<\infty}\to 1$ in probability as $n\to\infty$. Thus,
in this situation we cannot have a non-degenerate limit. The following example
fits into this regime. In this example the limits
\[
\gamma(t)\ :=\ \lim_{x\to\infty}\frac{L_t(x)}{\log x}
\ \in\ (0,\infty),\qquad t\ge 0,
\]
exist.
\begin{example}
Define $f(0):=0$, $f(1):=1$ and
\[
f(s)\ :=\ 1+\frac{s}{\log(1-s)},\qquad s\in (0,1).
\]
It is easily seen that $f$ has Taylor expansion $f(s)=\sum_{n\ge 1}p_ns^n$
with nonnegative coefficients
\[
p_n\ :=\ (-1)^{n-1}\int_0^1{x\choose n}\,{\rm d}x
\ =\ \frac{1}{n!}\int_0^1 x\frac{\Gamma(n-x)}{\Gamma(1-x)}\,{\rm d}x\ \ge\ 0,
\qquad n\in\nz.
\]
Thus, $f$ is the pgf of some random variable $\xi$ taking values in $\nz$.
From $p_0=0$ it follows that the associated continuous-time branching process
$Z=(Z_t)_{t\ge 0}$ has extinction probability $q=0$. Note that
$f(s)=1-(1-s)L((1-s)^{-1})$, where $L(x):=(x-1)/\log x$,
$x>1$, is regularly varying of index $1$.
For all $\varepsilon\in (q,1)=(0,1)$,
\begin{eqnarray*}
   \int_\varepsilon^1 \frac{1}{s-f(s)}\,{\rm d}s
   & = & \int_\varepsilon^1 \frac{1}{s-1-\frac{s}{\log(1-s)}}\,{\rm d}s
   \ = \ [\log(s+(1-s)\log(1-s))]_\varepsilon^1\\
   & = & -\log(\varepsilon+(1-\varepsilon)\log(1-\varepsilon))
   \ < \ \infty,
\end{eqnarray*}
which shows that $Z$ explodes. The Kolmogorov backward equation is
\begin{eqnarray*}
   at
   & = & \int_s^{F(s,t)}\frac{1}{f(u)-u}\,{\rm d}u
   \ = \ \int_s^{F(s,t)}\frac{1}{1-u+\frac{u}{\log(1-u)}}\,{\rm d}u\\
   & = & [-\log(u+(1-u)\log(1-u))]_s^{F(s,t)}\\
   & = & \log\frac{s+(1-s)\log(1-s)}{F(s,t)+(1-F(s,t))\log(1-F(s,t))}
\end{eqnarray*}
or, equivalently,
\[
F(s,t)+(1-F(s,t))\log(1-F(s,t))\ =\ e^{-at}(s+(1-s)\log(1-s))\ =:\ h(s,t).
\]
It is straightforward to check that this equation has the solution
\[
F(s,t)\ =\ 1-\exp\bigg(1+W\bigg(\frac{h(s,t)-1}{e}\bigg)\bigg),
\qquad s\in [0,1), t\ge 0,
\]
where $W=W_{-1}$ denotes the lower branch of the Lambert $W$ function
satisfying $W(h)e^{W(h)}=h$ and being real valued on $[-1/e,0)$.
Note that
$\pr(Z_t=\infty)=1-F(1,t)=\exp(1+W((e^{-at}-1)/e))
$
for $t\ge 0$, so $0<\pr(Z_t=\infty)<1$ for $t>0$. The time
$T:=\inf\{t>0:Z_t=\infty\}$ of explosion satisfies
$\pr(T<\infty)=\lim_{t\to\infty}\pr(Z_t=\infty)=\exp(1+W(-1/e))=\exp(0)=1$,
so $Z$ explodes in finite time almost surely. Note that $T$ has mean
\[
   \me(T)
   \ = \ \int_0^\infty \pr(Z_t<\infty)\,{\rm d}t
   \ = \ \int_0^\infty
         \bigg(
            1-\exp\bigg(1+W\bigg(\frac{e^{-at}-1}{e}\bigg)\bigg)
         \bigg)\,{\rm d}t.
\]
The substitution $x=1-e^{-at}$ ($\Rightarrow$ $t=-\frac{1}{a}\log(1-x)$ and
$\frac{{\rm d}t}{{\rm d}x}=\frac{1}{a(1-x)}$) leads to
\[
\me(T)\ =\
\frac{1}{a}\int_0^1\frac{1-\exp(1+W(-x/e))}{1-x}\,{\rm d}x.
\]
The function below the integral has a singularity at $x=1$.
From $1+W(-x/e)\sim \sqrt{2(1-x)}$ as $x\to 1$ it follows that the
function below the integral behaves asymptotically as
$\sqrt{2/(1-x)}$ as $x\to 1$, which yields $\me(T)<\infty$.\\
Let $G(s,t):=F(s,t)/F(1,t)$
denote the pgf of $Z_t$ conditioned on $Z_t<\infty$.
A somewhat tedious but straightforward calculation shows that
$1-G(s,t)\ =\ (1-s)L_t((1-s)^{-1})$, where $L_t$ is slowly
varying with
\[
\gamma(t)\ :=\ \lim_{x\to\infty}\frac{L_t(x)}{\log x}\ =\
\frac{w}{1+w-(w+1)^2e^{at}}
\]
with $w:=W(\frac{e^{-at}-1}{e})$. For $t\ge 0$ let
$(a_n(t))_{n\in\nz}$ be a sequence such that
$L_t(a_n(t))\sim a_n(t)/n$ as $n\to\infty$. Then, as explained before,
for every $t\ge 0$,
conditional on $Z_t<\infty$, $Z_t^{(n)}/a_n(t)\to 1$ in probability as
$n\to\infty$. A concrete sequence $(a_n(t))_{n\in\nz}$ is
$a_n(t):=\gamma(t)n\log n$, since, in this case,
$L_t(a_n(t))=L_t(\gamma(t)n\log n)\sim L_t(n\log n)
\sim\gamma(t)\log(n\log n)
\sim\gamma(t)\log n\ =\ a_n(t)/n$ as $n\to\infty$.
\end{example}
\subsection{Proof of Theorem \ref{main1}} \label{proofs}
The proof of Theorem \ref{main1} is quite natural and can be summarised
as follows. An application of the multivariate central limit theorem
yields the convergence of the finite-dimensional distributions. The
convergence in $D_\rz[0,\infty)$ is then established using a criterion
of Aldous \cite{aldous}. The following proof is relatively short and
elegant.
\begin{proof} (of Theorem \ref{main1})
   Let us compute for $s,t\ge 0$ the covariance of $Z_s$ and $Z_{s+t}$.
   For $k\in\nz_0$,
   \begin{eqnarray*}
      &   & \hspace{-15mm}\me((Z_s-m(s))(Z_{s+t}-m(s+t))\,|\,Z_s=k)\\
      & = & (k-m(s))\me(Z_{s+t}-m(s+t)\,|\,Z_s=k)
      \ = \ (k-m(s))\me(Z_t^{(k)}-m(s+t))\\
      & = & (k-m(s))(km(t)-m(s)m(t))
      \ = \ m(t)(k-m(s))^2.
   \end{eqnarray*}
   Thus, $\me((Z_s-m(s))(Z_{s+t}-m(s+t))\,|\,Z_s)=m(t)(Z_s-m(s))^2$
   almost surely. Taking expectation yields ${\rm Cov}(Z_s,Z_{s+t})
   =m(t){\rm Var}(Z_s)=m(t)\sigma^2(s)$.

   In order to verify the convergence $X^{(n)}\stackrel{\rm fd}{\to}X$
   of the finite-dimensional distributions fix $k\in\nz$ and
   $0\le t_1<\cdots<t_k<\infty$, define the $\rz^k$-valued random
   variable $Y:=(Z_{t_1}-m(t_1),\ldots,Z_{t_k}-m(t_k))$
   and let $Y_1,Y_2,\ldots$ be independent copies of $Y$. By the
   branching property, $(X_{t_1}^{(n)},\ldots,X_{t_k}^{(n)})
   =((Z_{t_1}^{(n)}-nm(t_1))/\sqrt{n},\ldots,(Z_{t_k}^{(n)}-nm(t_k))/\sqrt{n})$
   has the same distribution as $(Y_1+\cdots+Y_n)/\sqrt{n}$, which by the
   multivariate central limit theorem (see, for example,
   \cite[p.~16, Example 2.18]{vandervaart}) converges in distribution as
   $n\to\infty$ to a centered normal distribution $N(0,\Sigma)$ with
   covariance matrix $\Sigma=(\sigma_{i,j})_{1\le i,j\le k}$ having
   entries $\sigma_{i,j}:=\me((Z_{t_i}-m(t_i))(Z_{t_j}-m(t_j)))
   ={\rm Cov}(Z_{t_i},Z_{t_j})=m(|t_i-t_j|)\sigma^2(t_i\wedge t_j)$.
   Thus the convergence $X^{(n)}\stackrel{\rm fd}{\to} X$ of the
   finite-dimensional distributions holds.

   The convergence $X^{(n)}\to X$ in $D_\rz[0,\infty)$ is achieved as
   follows. Define the processes $M^{(n)}:=(M_t^{(n)})_{t\ge 0}$,
   $n\in\nz$, and $M:=(M_t)_{t\ge 0}$ via
   \[
   M_t^{(n)}\ :=\ \frac{X_t^{(n)}}{m(t)}
   \ =\ \sqrt{n}\bigg(\frac{Z_t^{(n)}}{nm(t)}-1\bigg)
   \quad\mbox{and}\quad
   M_t\ :=\ \frac{X_t}{m(t)},\qquad n\in\nz, t\ge 0.
   \]
   Then, $M,M^{(1)},M^{(2)},\ldots$ are martingales and $M$ is continuous,
   since the Gaussian process $X$ is continuous and $m(.)$ is continuous.
   Since $\me((M_t^{(n)})^2)={\rm Var}(M_t^{(n)})
   ={\rm Var}(Z_t^{(n)})/(n(m(t))^2)=\sigma^2(t)/(m(t))^2<\infty$
   does not depend on $n\in\nz$, we conclude that, for each $t\ge 0$, the
   family $\{M_t^{(n)}:n\in\nz\}$ is uniformly integrable. The convergence
   $M^{(n)}\to M$ in $D_\rz[0,\infty)$ therefore follows from Aldous'
   criterion \cite[Proposition 1.2]{aldous}. Since the map $t\mapsto m(t)$
   is continuous and deterministic it follows by multiplication with $m(t)$
   that $X^{(n)}\to X$ in $D_\rz[0,\infty)$.\hfill$\Box$
\end{proof}
\subsection{Proofs concerning Theorem \ref{main2}}
This section contains the proofs of Lemma \ref{lemmaoffspringdist} and
Theorem \ref{main2}.
\begin{proof} (of Lemma \ref{lemmaoffspringdist})
   The proof distinguishes the critical and non-critical case. Both cases
   are handled with different techniques.
   The representation in the critical case (for age-dependent
   branching processes) follows via an equivalence for the extinction
   probability from a combination of the results of Slack
   \cite[Theorem 1]{slack} and Vatutin \cite[Theorem 1]{vatutin2}.
   The following more elementary proof (see Case 1) is
   based on the backward equation and does not use extinction probabilities.

   \vspace{2mm}

   {\bf Case 1.} ($\lambda=0$) Let $t\ge 0$. In the critical case
   Kolmogorov's backward equation is
   \[
   at
   \ =\ \int_s^{F(s,t)}\frac{1}{f(x)-x}\,{\rm d}x
   \ =\ \int_s^{F(s,t)} \frac{1}{(1-x)^{\alpha}L((1-x)^{-1})}\,{\rm d}x,
   \qquad s\in [0,1].
	\]
   Since the map $x\mapsto f(x)-x$ is non-negative and non-increasing on
   $[0,1]$ it follows that
   \[
   \frac{F(s,t)-s}{(1-s)^{\alpha}L((1-s)^{-1})}
   \ \le\ at
   \ \le\ \frac{F(s,t)-s}{(1-F(s,t))^{\alpha}L((1-F(s,t))^{-1})}
   \]
   and, hence,
   \begin{eqnarray*}
      \limsup_{s\to 1}\frac{F(s,t)-s}{(1-s)^{\alpha}L((1-s)^{-1})}
      \ \le\ at
      & \le & \liminf_{s\to 1}
              \frac{F(s,t)-s}{(1-F(s,t))^{\alpha}L((1-F(s,t))^{-1})}\\
      & = & \liminf_{s\to 1} \frac{F(s,t)-s}{(1-s)^{\alpha}L((1-s)^{-1})},	
   \end{eqnarray*}
   where the last equality holds since $1-F(s,t)\sim 1-s$ as $s\to 1$. Thus,
   $\lim_{s\to 1}(F(s,t)-s)/((1-s)^{\alpha}L((1-s)^{-1}))=at$.

   \vspace{2mm}
	
   {\bf Case 2.} ($\lambda \ne 0$) Fix $t\ge 0$.
   Set $h_1(s):=(1-s)m(t)-(1-F(s,t))$ and
   $h_2(s):=(1-s)^\alpha L((1-s)^{-1})$ for $s\in [0,1)$.
   We have to verify that $\lim_{s\to 1}h_1(s)/h_2(s)=c(t)$, where $c(t)$
   is defined via (\ref{functionc}).
   By the Kolmogorov forward and backward equations,
   $h_1'(s)=-m(t)+\frac{\partial}{\partial s}F(s,t)=-m(t)+(f(F(s,t))-F(s,t))/(f(s)-s)$.
   Moreover,
   $h_2'(s)=(1-s)^{\alpha-1}L((1-s)^{-1})(L'((1-s)^{-1})(1-s)^{-1}/L((1-s)^{-1})-\alpha)$.
   From Assumption (\ref{assumption}), the asymptotics $1-F(s,t)\sim m(t)(1-s)$
   as $s\to 1$ and $(m(t))^\alpha=m(\alpha t)$ it follows that
   \begin{eqnarray}
      &   & \hspace{-10mm}m(\alpha t)-m(t)
      \ = \
		\lim_{s \to 1} \bigg( \frac{(1-F(s,t))m-(1-f(F(s,t)))}{(1-s)^{\alpha}L((1-s)^{-1})}
		- m(t)\frac{(1-s)m-(1-f(s))}{(1-s)^{\alpha}L((1-s)^{-1})} \bigg) \nonumber \\
		&=&
		\lim_{s \to 1} \bigg( (1-m)\frac{(1-s)m(t)-(1-F(s,t))}{(1-s)^{\alpha}L((1-s)^{-1})}
		\nonumber \\
		&& \hspace{1cm}+ \frac{m(t)(1-f(s)-(1-s))-(1-f(F(s,t)))+(1-F(s,t))}{(1-s)^{\alpha}L((1-s)^{-1})} \bigg) \nonumber \\
		&=&
		\lim_{s\to 1} \bigg( (1-m)\frac{(1-s)m(t)-(1-F(s,t))}{(1-s)^{\alpha}L((1-s)^{-1})}
+ \frac{-m(t)(f(s)-s)+f(F(s,t))-F(s,t)}{(1-s)^{\alpha}L((1-s)^{-1})} \bigg) \nonumber \\
		&=&
		(m-1)\lim_{s \to 1} \bigg(\alpha \frac{-m(t)(f(s)-s)+(f(F(s,t))-F(s,t))}{\alpha(m-1)(1-s)^{\alpha}L((1-s)^{-1})}
- \frac{(1-s)m(t)-(1-F(s,t))}{(1-s)^{\alpha}L((1-s)^{-1})} \bigg) \nonumber \\
      & = & (m-1)\lim_{s\to 1}
            \bigg(
		    \alpha\frac{h_1'(s)}{h_2'(s)+R(s)} - \frac{h_1(s)}{h_2(s)}
            \bigg).
            \label{local_7}
   \end{eqnarray}
   Using
   \[
   \frac{(1-m)(1-s)}{f(s)-s}
   \ =\ \frac{1-m}{1-m+(1-s)^{\alpha-1}L((1-s)^{-1})}
   \]	
   we see that $R(s)$ is given by
   \begin{eqnarray*}
      R(s)
      & = & -\alpha(1-s)^{\alpha-1}L((1-s)^{-1})
            \frac{1-m}{1-m + (1-s)^{\alpha-1}L((1-s)^{-1})}\\
      &   & \hspace{10mm}
            -(1-s)^{\alpha-1}L((1-s)^{-1})
            \bigg(\frac{L'((1-s)^{-1})(1-s)^{-1}}{L((1-s)^{-1})}-\alpha\bigg)\\
      & = & \alpha (1-s)^{\alpha-1} L((1-s)^{-1})
            \bigg(
               1-\frac{1-m}{1-m+(1-s)^{\alpha-1}L((1-s)^{-1})}\\
      &   &    \hspace{60mm}-\frac{L'((1-s)^{-1})(1-s)^{-1}}{\alpha L((1-s)^{-1})}
            \bigg).
   \end{eqnarray*}
   From Lamperti \cite[Theorem 2]{lamperti3} it follows that
   $\lim_{x\to\infty}xL'(x)/L(x)=0$. Applying this relation with
   $x:=(1-s)^{-1}$ yields
   \begin{equation}
      \lim_{s\to 1}\frac{R(s)}{h_2'(s)}
      \ =\ \lim_{s \to 1}
           \frac{\alpha\big(1-\frac{1-m}{1-m+(1-s)^{\alpha-1}L((1-s)^{-1})}-\frac{L'((1-s)^{-1})(1-s)^{-1}}{\alpha L((1-s)^{-1})}\big)}{\frac{L'((1-s)^{-1})(1-s)^{-1}}{L((1-s)^{-1})}-\alpha}
      \ =\ 0. \label{local_8}
   \end{equation}
   The three quantities $h_1(s)$, $h_2(s)$ and $(m(\alpha t)-m(t))/(m-1)$
   are non-negative, so from (\ref{local_7}) necessarily $\liminf_{s\to 1} h_1'(s)/(h_2'(s)+R(s))\ge 0$,
   leading to the boundary $h_1'(s)/(h_2'(s)+R(s))\ge(1-\delta)h_1'(s)/h_2'(s)$
   for any $0<\delta<(\alpha-1)/\alpha$ and $s$ sufficiently large. Then
   \[
   \frac{m(\alpha t)-m(t)}{m-1}
   \ \ge\ \limsup_{s\to 1}
          \bigg(\alpha(1-\delta)\frac{h_1'(s)}{h_2'(s)}-\frac{h_1(s)}{h_2(s)}\bigg),
   \]
   and the second part of Lemma \ref{lemmalhospital} provides
   \begin{equation}
      \limsup_{s \to 1}\frac{h_1(s)}{h_2(s)}
      \ \le\ \frac{m(\alpha t)-m(t)}{m-1}. \label{local_9}
   \end{equation}
   Now (\ref{local_7}), (\ref{local_8}) and (\ref{local_9}) yield
   \begin{eqnarray*}
      \frac{m(\alpha t)-m(t)}{m-1}
      & = & \lim_{s\to 1}
            \bigg(
               \bigg(\alpha\frac{h_1'(s)}{h_2'(s)}-\frac{h_1(s)}{h_2(s)}\bigg)
               \frac{h_2'(s)}{h_2'(s)+R(s)}
               - \frac{h_1(s)}{h_2(s)}\frac{R(s)}{h_2'(s)+R(s)}
           \bigg)\\
      & = & \lim_{s\to 1}
            \bigg(\alpha\frac{h_1'(s)}{h_2'(s)}-\frac{h_1(s)}{h_2(s)}\bigg).
	\end{eqnarray*}
   The claim follows again from Lemma \ref{lemmalhospital} in the appendix.
   Note that Lemma \ref{lemmalhospital} is applicable in both cases due to
   Lemma \ref{lemmaconvpgf}.\hfill$\Box$
\end{proof}
\begin{proof} (of Theorem \ref{main2})
   The proof is divided into four parts. The first part establishes the
   convergence of the one-dimensional distributions. The second and third
   part give two auxiliary results, one is about the normalizing sequence
   $(a_n)_{n\in\nz}$ and the other is a
   kind of upper bound 
   for the process, used in the final part to conclude the convergence in
   $D_\rz[0,\infty)$.

   \vspace{2mm}

   {\bf Part 1.} (Convergence of the one-dimensional distributions)

   \vspace{2mm}

   {\bf Version 1.} (based on $\alpha$-stable theory)
   Fix $t\in [0,\infty)$, define $Y:=Z_t$ for convenience and let
   $Y_1,Y_2,\ldots$ be independent copies of $Y$.

   Assume first that $\alpha\in (1,2)$. Then, by Bingham and Doney
   \cite[Theorem A]{binghamdoney}, Eq.~(\ref{betatdef}) is equivalent to
   $\pr(Y>x)\sim c(t)(-\Gamma(1-\alpha))^{-1}L(x)x^{-\alpha}$, $x\to\infty$.
   In particular, the map $x\mapsto \pr(Y>x)$ is regularly varying
   (at infinity) with index $-\alpha$. By Theorem 1 (ii) $\Rightarrow$ (i) of
   Geluk and de Haan \cite{gelukdehaan} (note that
   $p=1$ since $Y$ is nonnegative) it follows that the distribution function
   of $Y$ is in the domain of attraction of an $\alpha$-stable distribution,
   i.e. $\pr(Y\le .)\in D_\alpha$. The results at the top of
   p.~174 in \cite{gelukdehaan} on the
   choice of the normalizing sequences
   $(a_n)_{n\in\nz}$ and $(b_n)_{n\in\nz}$ furthermore
   show that, if we choose $a_n$ such that $L(a_n)\sim a_n^\alpha/(\alpha n)$
   as $n\to\infty$ and $b_n:=n\me(Y)/a_n=nm(t)/a_n$, then
   $(Z_t^{(n)}-nm(t))/a_n\stackrel{d}{=}(Y_1+\cdots+Y_n)/a_n-b_n\to X_t$
   in distribution as $n\to\infty$, where $X_t$ is $\alpha$-stable with
   characteristic function $u\mapsto\exp(c(t)(-iu)^\alpha/\alpha)$, $u\in\rz$.
   Thus, the convergence of the one-dimensional distributions holds.

   The case $\alpha=2$ is handled similarly by noting that
   (\ref{consequence}) is then equivalent (see \cite{binghamdoney}) to
   $\me(1_{\{Y\le x\}}Y^2)\sim 2c(t)L(x)$ as
   $x\to\infty$ such that we can apply Theorem 2
   of Geluk and de Haan \cite{gelukdehaan}.

   \vspace{2mm}

   {\bf Version 2.} (based on Laplace transforms)
   Fix $t\in [0,\infty)$. For every $n\in\nz$ the real valued random
   variable $X_t^{(n)}$ has Laplace transform
   \begin{eqnarray*}
      \eta\ \mapsto\ \me(\exp(-\eta X_t^{(n)}))
      & = & \me(\exp(-\eta a_n^{-1}(Z_t^{(n)}-nm(t))))
      \ = \ s_n^{-nm(t)} (F(s_n,t))^n,\quad \eta\ge 0,
   \end{eqnarray*}
   where $s_n:=\exp(-\eta/a_n)$. In order to verify that
   $\lim_{n\to\infty}\me(\exp(-\eta X_t^{(n)}))=\me(\exp(-\eta X_t))$
   assume without loss of generality that $\eta>0$.
   Taking logarithm yields
   \begin{equation} \label{log}
      \log \me(\exp(-\eta X_t^{(n)}))
      \ = \ -nm(t)\log s_n + n\log F(s_n,t)
      \ = \ \eta m(t)\frac{n}{a_n} - nx_n + O(nx_n^2),
   \end{equation}
   where $x_n:=1-F(s_n,t)$. Note that $s_n\to 1$ and, hence,
   $x_n\to 0$ as $n\to\infty$. More precisely, by assumption,
   \[
   x_n
   \ =\ m(t)(1-s_n) - c(t)(1-s_n)^\alpha L((1-s_n)^{-1})(1+o(1)),
   \qquad n\to\infty.
   \]
   From $1-s_n=1-\exp(-\eta/a_n)=\eta/a_n+O(1/a_n^2)$ and
   $L((1-s_n)^{-1})\sim L(a_n/\eta)\sim L(a_n)\sim a_n^\alpha/(\alpha n)$
   we conclude that
   \begin{eqnarray*}
      x_n
      & = & m(t)\bigg(\frac{\eta}{a_n}+O\bigg(\frac{1}{a_n^2}\bigg)\bigg)
            - c(t)\bigg(\bigg(\frac{\eta}{a_n}\bigg)^\alpha + O\bigg(\frac{1}{a_n^{\alpha+1}}\bigg)\bigg)
            \cdot \frac{a_n^\alpha}{\alpha n}(1+o(1))\\
      & = & \frac{\eta m(t)}{a_n}
            -c(t)\frac{\eta^\alpha}{\alpha n} + O\bigg(\frac{1}{a_n^2}\bigg) + o\bigg(\frac{1}{n}\bigg)
      \ = \ \frac{\eta m(t)}{a_n}-c(t)\frac{\eta^\alpha}{\alpha n}
            +o\bigg(\frac{1}{n}\bigg),
   \end{eqnarray*}
   since
   $n/a_n^2\sim (\alpha L(a_n)a_n^{2-\alpha})^{-1}\to 0$ as $n\to\infty$.
   It follows that (\ref{log}) converges to $c(t)\eta^\alpha/\alpha
   =\log\me(\exp(-\eta X_t))$ as $n\to\infty$.
   Thus, the Laplace transform of $X_t^{(n)}$ converges pointwise
   on $[0,\infty)$ to the Laplace transform of $X_t$. In other words,
   the moment generating function of $X_t^{(n)}$ is finite on the interval
   $I:=(-\infty,0]$ and converges pointwise on $I$
   as $n\to\infty$ to the moment generating function of $X_t$. This
   implies (see, for example, Billingsley
   \cite[p.~397, Problem 30.4]{billingsley} or Kallenberg \cite[p.~101,
   Exercise 9]{kallenberg}) the convergence $X_t^{(n)}\to X_t$ in
   distribution as $n\to\infty$.

   \vspace{2mm}

   {\bf Part 2.} (Asymptotic relation for $(a_n)_{n\in\nz}$)
   Let $(\varepsilon_n)_{n\in\nz}$ be an arbitrary sequence of positive
   real numbers converging to zero as $n\to\infty$. For $n\in\nz$ and $T>0$
   define $S_{n,T}:=[-\varepsilon_nn/a_n,\varepsilon_nn/a_n]\times [0,T]$,
   where $(a_n)_{n\in\nz}$ is the normalizing
   sequence satisfying $a_n/(L(a_n))^{1/\alpha}\sim(\alpha n)^{1/\alpha}$
   as $n\to\infty$.
   Bojanic and Seneta \cite[p.~308]{bojanicseneta} provide the
   existence of another slowly varying function $L^\ast$ such that
   $a_n\sim(\alpha n)^{1/\alpha}L^{\ast}(n^{1/\alpha})$ as $n\to\infty$.
   Set $h(n):=(\alpha n)^{1/\alpha}L^{\ast}(n^{1/\alpha})/a_n$ for $n\in\nz$
   and $h(r):=h(\lfloor r\rfloor)$ for $r\in\rz$, $r\ge 1$. Then the
   asymptotic relation simply means $\lim_{r\to\infty}h(r)=1$. From
   \begin{eqnarray}
      \lim_{n\to\infty}\inf_{(x,s)\in S_{n,T}}(nm(s)+xa_n)\ =\ \infty
      \label{eq_infconv}
   \end{eqnarray}
   it follows that
   $\sup_{(x,s)\in S_{n,T}}|h(nm(s)+xa_n)-1|\to 0$ as
   $n$ tends to infinity. Furthermore, $\lim_{n\to\infty}
   \sup_{(x,s)\in S_{n,T}}|xa_n/n|\le\lim_{n\to\infty}\varepsilon_n=0$
   implies $\lim_{n\to\infty}\sup_{(x,s)\in S_{n,T}}
   |(m(s)+xa_n/n)^{1/\alpha}-(m(s))^{1/\alpha}|=0$ as
   well as, using the uniform convergence theorem for slowly varying
   functions (see, for example, Bingham, Goldie and Teugels
   \cite[Theorem 1.2.1]{binghamgoldieteugels} or Bojanic and Seneta
   \cite{bojanicseneta})
   \begin{eqnarray*}
      \lim_{n\to\infty}\sup_{(x,s)\in S_{n,T}}
      \bigg|
         \frac{L^\ast(n^{1/\alpha}(m(s)+xa_n/n)^{1/\alpha})}{L^\ast(n^{1/\alpha})}
         -1
      \bigg|
      \ =\ 0.
   \end{eqnarray*}
   Having bounded limits, the listed uniformly convergent sequences
   are uniformly bounded and thus their product converges again
   uniformly, yielding
   \begin{eqnarray}
      &   & \hspace{-10mm}\lim_{n\to\infty}\sup_{(x,s)\in S_{n,T}}
	        \bigg|\frac{a_{nm(s)+xa_n}}{a_n}
            - (m(s))^{1/\alpha}\bigg|\nonumber\\
      & = & \lim_{n\to\infty}\sup_{(x,s)\in S_{n,T}}
            \bigg|
               \frac{h(n)}{h(nm(s)+xa_n)}\frac{L^\ast((nm(s)+xa_n)^{1/\alpha})}{L^\ast(n^{1/\alpha})}
               \bigg(m(s)+\frac{xa_n}{n}\bigg)^{1/\alpha} - (m(s))^{1/\alpha}
            \bigg|\nonumber\\
	  & = & 0.
   \label{normalizingconstconv}
   \end{eqnarray}

   \vspace{2mm}

   {\bf Part 3.} (Kind of upper bound for $X_t^{(n)}$)
   In this part it is shown that for each $T>0$ there exists a sequence
   $(\varepsilon_n)_{n\in\nz}$ of positive real numbers with
   $\lim_{n\to\infty}\varepsilon_n=0$ such that
   \begin{equation}
      \lim_{n\to\infty}
      \pr\bigg(
      \sup_{t\in [0,T]} |X_t^{(n)}| \ge\frac{\varepsilon_nn}{a_n}
      \bigg)
      \ =\ 0.
   \label{pwbound}
   \end{equation}
   Let $\delta:=0$ if $m<1$ and $\delta:=T$ if $m\ge 1$. Then,
   for any sequence $(\varepsilon_n)_{n\in\nz}$ of positive real numbers,
   \[
   \pr\bigg(\sup_{t\in [0,T]}|X_t^{(n)}|\ge\frac{\varepsilon_nn}{a_n}\bigg)
   \ \le\
   \pr\bigg(
   \sup_{t\in [0,T]} \bigg|\frac{b_nX_t^{(n)}}{m(t)}\bigg| \ge \frac{\varepsilon_n n}{m(\delta)}
   \bigg).
   \]
   Applying Doob's submartingale inequality to the martingale
   $(a_nX_t^{(n)}/m(t))_{t\ge 0}=(Z_t^{(n)}/m(t)-n)_{t\ge 0}$ yields
   \begin{eqnarray*}
	   \pr\bigg( \sup_{t\in [0,T]}\bigg|\frac{a_nX_t^{(n)}}{m(t)}\bigg|
       \ge \frac{\varepsilon_n n}{m(\delta)}\bigg)
	   & \le & \frac{m(\delta)}{\varepsilon_nn}
               \me\bigg(\bigg| \frac{Z_T^{(n)}}{m(T)}-n\bigg|\bigg)
       \ =\ \frac{m(\delta)}{m(T)}
             \frac{1}{\varepsilon_n}
             \me\bigg(\bigg|\frac{Z_T^{(n)}}{n}-m(T)\bigg|\bigg).
   \end{eqnarray*}
   By the law of large numbers the latter expectation converges to
   $0$ as $n\to\infty$. Thus the sequence $(\varepsilon_n)_{n\in\nz}$
   can be chosen such that $\lim_{n\to\infty}\varepsilon_n=0$ and such
   that the right-hand side still converges to $0$, which implies that
   (\ref{pwbound}) holds for the particular sequence $(\varepsilon_n)_{n\in\nz}$.

   \vspace{2mm}

   {\bf Part 4.} (Convergence in $D_{\rz}[0,\infty)$)
   In general, the processes $X^{(n)}$ and $X$ are time-in\-ho\-mo\-geneous.
   Let $Y^{(n)}:=(X_t^{(n)},t)_{t\ge 0}$ and $Y:=(X_t,t)_{t\ge 0}$
   denote the space-time processes of $X^{(n)}$ and $X$ respectively.
   According to Revuz and Yor \cite[p.~85, Exercise (1.10)]{revuzyor}
   the processes $Y^{(n)}$ and $Y$ are time-homogeneous Markov processes
   with state space $S:=\rz\times [0,\infty)$. Recall that
   $S_{n,T}=[-\varepsilon_nn/a_n,\varepsilon_nn/a_n]\times[0,T]$,
   where $(\varepsilon_n)_{n\in\nz}$ is the sequence defined in Part 4.
   In terms of $Y^{(n)}$, (\ref{pwbound}) is simply
   \begin{equation}
      \lim_{n\to\infty}\pr\big(Y_t^{(n)}\in S_{n,T}, 0\le t\le T\big)\ =\ 1.
      \label{local_01}
   \end{equation}
   Corollary 8.7 on p.~232 of Ethier and Kurtz \cite{ethierkurtz} states
   that (\ref{local_01}) jointly with the uniform convergence of the
   semigroups on the restricted area $S_{n,T}$ implies the convergence
   of $Y^{(n)}$ to $Y$ in $D_S[0,\infty)$, hence the desired convergence
   of $X^{(n)}$ to $X$ in $D_{\rz}[0,\infty)$. Thus it remains to show
   that for each $f\in\widehat{C}(S)$, the space of
   real valued continuous functions on $S$ vanishing at infinity, and $t\in[0,T]$
   \begin{eqnarray}
      \lim_{n\to\infty}\sup_{(x,s)\in S_{n,T}}
      |\widetilde{T}_t^{(n)}f(x,s)-\widetilde{T}_tf(x,s)|
      \ =\ 0, \label{semigroupconv}
   \end{eqnarray}
   where $(\widetilde{T}_t^{(n)})_{t\ge 0}$ and $(\widetilde{T}_t)_{t\ge 0}$
   denote the semigroups of $Y^{(n)}$ and $Y$ respectively, that is
   $\widetilde{T}_t^{(n)}f(x,s)=\me(f(X_{s+t}^{(n)},s+t)\,|\,X_s^{(n)}=x)$
   and $\widetilde{T}_tf(x,s)=\me(f(X_{s+t},s+t)\,|\,X_s=x)$
   for all $f\in\widehat{C}(S)$ and $(x,s)\in S$. By Lemma \ref{dense}
   the space of all maps of the form
   $(x,s)\mapsto\sum_{i=1}^{l} g_i(x)h_i(s)$ with $l\in\nz$,
   $g_i\in\widehat{C}(\rz)$ and $h_i\in\widehat{C}([0,\infty))$ is dense
   in $\widehat{C}(S)$. Hence it suffices to show (\ref{semigroupconv})
   for $f = gh$ with $g\in\widehat{C}(\rz)$ and $h\in\widehat{C}([0,\infty))$,
   in which case
   \[
   \widetilde{T}_t^{(n)}f(x,s)
   \ =\ h(s+t)\me(g(X_{s+t}^{(n)})\,|\,X_s^{(n)}=x)
   \ =\ h(s+t)\me\bigg(g\bigg(\frac{a_k}{a_n}X_t^{(k)}+xm(t)\bigg)\bigg),
   \quad (x,s)\in S,
   \]
   where $k:=k(n,s,x):=nm(s)+xa_n$, and
   \[
   \widetilde{T}_tf(x,s)
   \ =\ h(s+t)\me(g(X_{s+t}),|\,X_s=x)
   \ =\ h(s+t)\me(g(m(s)^{1/\alpha}X_t+xm(t))),
   \qquad (x,s)\in S.
   \]
   Let $\varepsilon>0$. Choose $C>0$ such that
   $\sup_{n\in\nz}\pr(|X_t^{(n)}|>C)<\varepsilon$.
   Splitting the mean along the event $A_k:=\{|X_t^{(k)}|\le C\}$ yields
   \begin{eqnarray*}
      &   & \hspace{-10mm}\sup_{(x,s)\in S_{n,T}}
	        |\widetilde{T}_t^{(n)}f(x,s)-\widetilde{T}_tf(x,s)|\\
      & = & \sup_{(x,s)\in S_{n,T}} h(s+t)
            \bigg|
            \me\bigg(g\bigg(\frac{a_k}{a_n}X_t^{(k)}+xm(t)\bigg)\bigg) -
            \me\big(g\big((m(s))^{1/\alpha}X_t + xm(t)\big)\big)
            \bigg|\\
      & \le & \|h\|\bigg(
            \sup_{(x,s)\in S_{n,T}}
	        \big|
               \me\big(g((m(s))^{1/\alpha}X_t^{(k)}+xm(t))\big)
	           -\me\big(g\big((m(s))^{1/\alpha}X_t+xm(t)\big)\big)
            \big|  \\
      &   & ~~+2\|g\|\varepsilon
            + \sup_{(x,s)\in S_{n,T}}
            \me\bigg(
               1_{A_k}\bigg|
                  g\bigg(\frac{a_k}{a_n}X_t^{(k)}+xm(t)\bigg)
                  - g\big((m(s))^{1/\alpha}X_t^{(k)} + xm(t)\big)
               \bigg|
            \bigg)
        \bigg).
   \end{eqnarray*}
   The second last supremum converges to $0$ as $n\to\infty$ by Lemma
   \ref{lemmauniformweakconv} and since $k\to\infty$ as $n\to\infty$
   by (\ref{eq_infconv}). The last supremum converges as well to $0$
   by (\ref{normalizingconstconv}) together with the uniform continuity
   of $g$. Since $\varepsilon >0$ can be chosen arbitrarily,
   (\ref{semigroupconv}) holds, which completes the proof.\hfill$\Box$
\end{proof}
\subsection{Proofs concerning Theorem \ref{main3}} \label{proofs3}
This section contains the proofs of Lemma \ref{lemmaminfty}, Lemma
\ref{betatlemma} and Theorem \ref{main3}.
\begin{proof} (of Lemma \ref{lemmaminfty})
   Fix $t\ge 0$. By Theorem 2 or Corollary 2.2 of Lamperti
   \cite{lamperti3}, applied with
   $x:=1-s$ to the function $x\mapsto 1-F(1-x,t)$, (\ref{eqgeneratingfunc})
   holds if and only if
   \begin{eqnarray}
      \lim_{s\to 1}\alpha(s,t)\ =\ \alpha(t),
	  \label{eqlampertiforF}
   \end{eqnarray}	
   where
   \[
   \alpha(s,t)
   \ :=\ \frac{(1-s)\frac{\partial}{\partial s} F(s,t)}{1-F(s,t)}
   \ =\ \frac{f(F(s,t))-F(s,t)}{1-F(s,t)}\frac{1-s}{f(s)-s}\\
   \ =\ \frac{L((1-F(s,t))^{-1})-1}{L((1-s)^{-1})-1}
   \]
   for all $s\in (0,1)$.
   Thus (i) and (ii) are equivalent. By Kolmogorov's backward equation,
   \begin{equation}
      at
      \ =\ \int_s^{F(s,t)}\frac{1}{f(u)-u}\,{\rm d}u
      \ =\ \int_{(1-F(s,t))^{-1}}^{(1-s)^{-1}}\frac{1}{x(L(x)-1)}\,{\rm d}x.
	   \label{eqkolmogorov}
   \end{equation}
   Also, note that
   \[
   \log\frac{1}{\alpha(s,t)}
   \ =\ \log (L((1-s)^{-1})-1)-\log(L((1-F(s,t))^{-1})-1)
   \ =\ \int_{(1-F(s,t))^{-1}}^{(1-s)^{-1}}\frac{L'(x)}{L(x)-1}\,{\rm d}x.
   \]	
   (iii) $\Rightarrow$ (ii):
   Applying integration by parts to (\ref{eqkolmogorov}) yields
   \begin{equation}
      at\ =\
      \frac{\log x}{L(x)-1}
      \bigg\vert_{x =(1-F(s,t))^{-1}}^{x=(1-s)^{-1}}
      + \int_{(1-F(s,t))^{-1}}^{(1-s)^{-1}}
      \frac{\log x}{L(x)-1}\frac{L'(x)}{L(x)-1}\,{\rm d}x.
      \label{eqproof_intbyparts}
   \end{equation}
   In the following we distinguish the two cases $A>0$ and $A=0$.
   Assume first that $L(x)/\log x\to A$ as $x\to\infty$ for some
   $A>0$ and let $t>0$. Let $\varepsilon>0$ be arbitrary. Then
   there exists $K>0$ such that $1-\varepsilon\le A\log x /(L(x)-1)\le 1+\varepsilon$
   for all $x\ge K$. But, if $s$ is sufficiently close to $1$, both
   inequalities hold on the interval where it is integrated above
   in (\ref{eqproof_intbyparts}), implying that
   $Aat=\lim_{s\to 1}\log(\alpha(s,t))^{-1}$, which is exactly
   (\ref{eqlampertiforF}).

   (i) $\Rightarrow$ (iii):
   Assume that (\ref{eqgeneratingfunc}) holds for all $t\ge 0$.
   By (\ref{alphat}),
   \[
   \alpha(t)\ =\ \lim_{s\to 1}\frac{\log(1-F(s,t))}{\log(1-s)}.
   \]
%
   As already seen before Lemma \ref{lemmaminfty} there exists
   $C\ge 0$ such that $\alpha(t)=e^{-Ct}$. Thus,
   \begin{equation}
      Ct
      \ =\ -\lim_{s\to 1}\log\frac{\log (1-F(s,t))}{\log(1-s)}
      \ =\ \lim_{s\to 1}
           \int_{(1-F(s,t))^{-1}}^{(1-s)^{-1}}\frac{1}{x\log x}\,{\rm d}x.
      \label{eqintegralone_local}
   \end{equation}
   Division of (\ref{eqintegralone_local}) by (\ref{eqkolmogorov}) leads to
   \begin{eqnarray*}
      A\ :=\ \frac{C}{a}\ =\ \lim_{s\to 1}
      \frac{\int_{(1-F(s,t))^{-1}}^{(1-s)^{-1}}\frac{1}{x\log x}\,{\rm d}x}
           {\int_{(1-F(s,t))^{-1}}^{(1-s)^{-1}}\frac{1}{x(L(x)-1)}\,{\rm d}x}.
   \end{eqnarray*}
   Now exploit the monotonicity of $\log x$ and $L(x)$ to conclude that
   \begin{eqnarray*}
      A
      & \le & \liminf_{s\to 1}\frac{\frac{1}{\log( (1-F(s,t))^{-1})}
              \int_{(1-F(s,t))^{-1}}^{(1-s)^{-1}}\frac{1}{x}\,{\rm d}x}
              {\frac{1}{L((1-s)^{-1})-1}
              \int_{(1-F(s,t))^{-1}}^{(1-s)^{-1}}\frac{1}{x}\,{\rm d}x}\\
      & = & \liminf_{s\to 1}\frac{L((1-s)^{-1})}{\log ((1-F(s,t))^{-1})}
      \ = \ \frac{1}{\alpha(t)}
            \liminf_{s\to 1}\frac{L((1-s)^{-1})}{\log((1-s)^{-1})}.
   \end{eqnarray*}
   Similarly, $A\ge\alpha(t)\limsup_{s\to 1} L((1-s)^{-1})/\log((1-s)^{-1})$.
   Letting $t\to 0$ yields $A=\lim_{s\to 1} L((1-s)^{-1})/\log((1-s)^{-1})$,
   which is (iii) and completes the proof.\hfill$\Box$
\end{proof}
\begin{proof} (of Lemma \ref{betatlemma})
   By assumption, $H(x):=L(x)-1-A\log x$, $x\ge 1$, satisfies
   $\lim_{x\to\infty}H(x)=B-1$. Moreover, $\beta(t)$, defined via
   (\ref{betatformula}), satisfies
   \begin{equation}
      \log\beta(t)
      \ =\ a\int_0^t(B-1-A\log\beta(s))\,{\rm d}s,\qquad t\ge 0.
      \label{eqintegralct}
   \end{equation}
   Computing the derivative of $L_t(x)$ with respect to $t$ provides
   a representation for $L_t(x)$ similar to (\ref{eqintegralct}), namely
   \begin{eqnarray*}
      &   & \hspace{-15mm}\frac{\partial}{\partial t}L_t(x)
      \ = \ \frac{\partial}{\partial t}
            \big(x^{\alpha(t)}(1-F(1-x^{-1},t))\big)\\
      & = & x^{\alpha(t)}\alpha'(t)\log x(1-F(1-x^{-1},t))
            - x^{\alpha(t)}a(f(F(1-x^{-1},t))-F(1-x^{-1},t))\\
      & = & ax^{\alpha(t)}(1-F(1-x^{-1},t))\bigg(
            \frac{1-f(F(1-x^{-1},t))}{1-F(1-x^{-1},t)}
            - \frac{1-F(1-x^{-1},t)}{1-F(1-x^{-1},t)}
            - A\alpha(t)\log x
            \bigg)\\
      & = & aL_t(x)\Big(
               L((1-F(1-x^{-1},t))^{-1}) -1 - A\log x^{\alpha(t)}
            \Big)\\
      & = & aL_t(x)\Big(
               L(x^{\alpha(t)}L_t^{-1}(x)) -1 - A\log x^{\alpha(t)}
            \Big)\\
      & = & aL_t(x)\Big(H(x^{\alpha(t)}L_t^{-1}(x))-A\log L_t(x)\Big),
            \qquad t\ge 0,x\ge 1.
   \end{eqnarray*}
   Therefore
   \begin{equation}
      \log L_t(x)
      \ = \ \int_0^t \frac{\frac{\partial}{\partial s}L_s(x)}{L_s(x)}\,{\rm d}s
      \ = \ a\int_0^t \big(
               H(x^{\alpha(s)}L_s^{-1}(x))-A\log L_s(x)
            \big)\,{\rm d}s,\qquad t\ge 0.
      \label{eqintegralLt}
   \end{equation}
   Let $t>0$ be fixed and $\varepsilon>0$ be arbitrary. If $1-x^{-1}>q$,
   then the map $s\to x^{\alpha(s)}L_s^{-1}(x)=(1-F(1-x^{-1},s))^{-1}$ is
   non-increasing. Hence $|H(x^{\alpha(s)}L_s^{-1}(x))-(B-1)|<\varepsilon$
   for all $s\in[0,t]$ and all sufficiently large $x$. From
   (\ref{eqintegralct}) and (\ref{eqintegralLt}) we obtain
   \[
   |\log L_t(x)-\log\beta(t)|\ \le\
   a\varepsilon t + aA\int_0^t |\log L_s(x)-\log\beta(s)|\,{\rm d}s.
   \]
   By Gronwall's inequality,
   \begin{eqnarray*}
      |\log L_t(x)-\log\beta(t)|
      & \le & a\varepsilon t + aA\int_0^t a\varepsilon s
                \exp\bigg(\int_s^taA\,{\rm d}\sigma\bigg)\,{\rm d}s\\
      & \le & a\varepsilon t\bigg(
              1 + \int_0^t aA \exp(aA(t-s))\,{\rm d}s \bigg)
	  \ =\ a\varepsilon t\exp(aAt).
   \end{eqnarray*}
   Since $\varepsilon>0$ can be chosen arbitrarily small, the result
   $\lim_{x\to\infty}L_t(x)=\beta(t)$ follows.\hfill$\Box$
\end{proof}
\begin{proof} (of Theorem \ref{main3})
   The proof is divided into two steps. First the assumption
   (\ref{betatdef}) is used to establish the convergence of the
   one-dimensional distributions. Afterwards it is shown with some general
   weak convergence machinery for Markov processes that the convergence
   of the one-dimensional distributions is already sufficient for
   convergence in $D_E[0,\infty)$, where $E:=[0,\infty)$.

   \vspace{2mm}

   {\bf Step 1.} (Convergence of the one-dimensional distributions)
   Fix $\lambda,t\ge 0$. Define $s_n:=\exp(-\lambda n^{-1/\alpha(t)})$,
   $n\in\nz$. Note that $s_n\to 1$ as $n\to\infty$. We have
   $\me(\exp(-\lambda X_t^{(n)}))
       =  \me(\exp(-\lambda n^{-1/\alpha(t)}Z_t^{(n)}))
       =  (\me(\exp(-\lambda n^{-1/\alpha(t)}Z_t)))^n
       =  (F(s_n,t))^n$.
   Taking the logarithm yields
   \[
   \log\me(\exp(-\lambda X_t^{(n)}))
   \ =\ n\log (1-(1-F(s_n,t)))
   \ \sim\ -n(1-F(s_n,t))
   \ \sim\ -n\beta(t)(1-s_n)^{\alpha(t)}
   \]
   as $n\to\infty$ by (\ref{betatdef}). Since
   $1-s_n=1-\exp(-\lambda n^{-1/\alpha(t)})
   \sim \lambda n^{-1/\alpha(t)}$ as $n\to\infty$ it follows that
   the latter expression is asymptotically equal to
   $-n\beta(t)(\lambda n^{-1/\alpha(t)})^{\alpha(t)}=-\beta(t)\lambda^{\alpha(t)}$.
   Therefore $\lim_{n\to\infty}\me(\exp(-\lambda X_t^{(n)}))
   =\exp(-\beta(t)\lambda^{\alpha(t)})=\me(\exp(-\lambda X_t))$.
   This pointwise convergence of the Laplace
   transforms implies the convergence $X_t^{(n)}\to X_t$ in distribution
   as $n\to\infty$.

   \vspace{2mm}

   {\bf Step 2.} (Convergence in $D_E[0,\infty)$)
   We proceed as in the proof of \cite[Theorem 2.1]{kuklamoehle}.
   For $n\in\nz$ and $t\ge 0$ define $E_{n,t}:=\{j/n^{1/\alpha(t)}\,:\,j\in\nz_0\}$.
   In general the process $X^{(n)}$ is time-inhomogeneous. Let
   $Y^{(n)}:=(X_t^{(n)},t)_{t\ge 0}$ and $Y:=(X_t,t)_{t\ge 0}$ denote
   the space-time processes of $X^{(n)}$ and $X$
   respectively. Note that $Y^{(n)}$ has state space
   $S_n:=\{(j/n^{1/\alpha(t)},t)\,:\,j\in\nz_0,t\ge 0\}
   =\bigcup_{t\ge 0}(E_{n,t}\times\{t\})$ and $Y$ has state space
   $S:=[0,\infty)^2$. According to Revuz and Yor
   \cite[p.~85, Exercise (1.10)]{revuzyor} the process $Y^{(n)}$ is
   time-homogeneous. Define $\pi_n:B(S)\to B(S_n)$ via
   $\pi_nf(x,s):=f(x,s)$ for $f\in B(S)$ and $(x,s)\in S_n$.
   In the following it is shown that $Y^{(n)}$ converges in $D_S[0,\infty)$
   to $Y$ as $n\to\infty$. Note that this convergence implies the desired
   convergence of $X^{(n)}$ in $D_E[0,\infty)$ to $X$ as $n\to\infty$.
   For $\lambda,\mu>0$ define the test function $f_{\lambda,\mu}$
   via $f_{\lambda,\mu}(x,s):=e^{-\lambda x-\mu s}$, $(x,s)\in S$.
   By \cite[Proposition 5.4]{kuklamoehle} it suffices to verify that
   for every $t\ge 0$ and $\lambda,\mu>0$,
   \begin{equation} \label{conv}
      \lim_{n\to\infty}\sup_{s\ge 0}\sup_{x\in E_{n,s}}
      |U_t^{(n)}\pi_nf_{\lambda,\mu}(x,s)-\pi_nU_tf_{\lambda,\mu}(x,s)|
      \ =\ 0,
   \end{equation}
   where $U_t^{(n)}:B(S_n)\to B(S_n)$ is defined via
   $U_t^{(n)}f(x,s):=\me(f(X_{s+t}^{(n)},s+t)\,|\,X_s^{(n)}=x)$,
   $f\in B(S_n)$, $s\ge 0$, $x\in E_{n,s}$. Note that
   $(U_t^{(n)})_{t\ge 0}$ is the semigroup of $Y^{(n)}$.

   Fix $t\ge 0$ and $\lambda,\mu>0$. For all $n\in\nz$, $s\ge 0$ and
   $x\in E_{n,s}$,
   \begin{eqnarray*}
      U_t^{(n)}\pi_nf_{\lambda,\mu}(x,s)
      & = & \me(\pi_nf_{\lambda,\mu}(X_{s+t}^{(n)},s+t)\,|\,X_s^{(n)}=x)\\
      & = & \me(\exp(-\lambda X_{s+t}^{(n)}-\mu(s+t))\,|\,X_s^{(n)}=x)\\
      & = & e^{-\mu(s+t)}\me(\exp(-\lambda n^{-1/\alpha(s+t)}
            Z_{s+t}^{(n)})\,|\,Z_s^{(n)}=xn^{1/\alpha(s)})\\
      & = & e^{-\mu(s+t)}\me(\exp(-\lambda n^{-1/\alpha(s+t)}Z_t^{(xn^{1/\alpha(s)})}))
   \end{eqnarray*}
   and
   \begin{eqnarray*}
      \pi_nU_tf_{\lambda,\mu}(x,s)
      & = & U_tf_{\lambda,\mu}(x,s)
      \ = \ \me(\exp(-\lambda X_{s+t}-\mu(s+t))\,|\,X_s=x)\\
      & = & e^{-\mu(s+t)}\me(\exp(-\lambda X_{s+t})\,|\,X_s=x)
      \ = \ e^{-\mu(s+t)}\me(\exp(-\lambda x^{1/\alpha(t)}X_t)).
   \end{eqnarray*}
   Thus, one has to verify that
   \[
   \lim_{n\to\infty}\sup_{s\ge 0}\sup_{x\in E_{n,s}}
   e^{-\mu(s+t)}
   |\me(\exp(-\lambda n^{-1/\alpha(s+t)}Z_t^{(xn^{1/\alpha(s)})}))
   - \me(\exp(-\lambda x^{1/\alpha(t)}X_t))|
   \ =\ 0.
   \]
   We will even verify that
   \[
   \lim_{n\to\infty}\sup_{s\ge 0}\sup_{x>0}
   |\me(\exp(-\lambda n^{-1/\alpha(s+t)}Z_t^{(\lfloor xn^{1/\alpha(s)}\rfloor)}))
   - \me(\exp(-\lambda x^{1/\alpha(t)}X_t))|
   \ =\ 0.
   \]
   Since $\alpha(s+t)=\alpha(s)\alpha(t)$,
   the quantity inside the absolute values depends
   on $n$ and $s$ only via $n^{1/\alpha(s)}$. Since
   $n^{1/\alpha(s)}$ is non-decreasing in $s$ it follows that the convergence
   for fixed $s\ge 0$ is slower as $s$ is smaller. So the slowest convergence
   holds for $s=0$ ($\Rightarrow\alpha(s)=1$). Thus it suffices to verify that
   for every $t\ge 0$ and $\lambda>0$
   \[
   \lim_{n\to\infty}\sup_{x>0}
   |\me(\exp(-\lambda n^{-1/\alpha(t)}Z_t^{(\lfloor xn\rfloor)}))
   - \me(\exp(-\lambda x^{1/\alpha(t)}X_t))|
   \ =\ 0.
   \]
   The map $x\mapsto\me(\exp(-\lambda x^{1/\alpha(t)}X_t))$ is bounded,
   continuous and non-increasing. Since $Z_t^{(1)}\le Z_t^{(2)}\le\cdots$
   almost surely it follows by P\'olya's theorem \cite[Satz I]{polya} that
   it suffices to verify the above convergence pointwise for every $x>0$.
   Defining $k:=\lfloor xn\rfloor$ it is readily seen that this is
   equivalent to the convergence of the one-dimensional distributions
   $X_t^{(k)}=k^{-1/\alpha(t)}Z_t^{(k)}\to X_t$ in distribution as
   $k\to\infty$, $t\ge 0$. But the convergence of the one-dimensional
   distributions holds by Step 1. The proof is complete.\hfill$\Box$
\end{proof}

\subsection{Appendix} \label{appendix}
In this appendix four auxiliary results are provided. Lemma
\ref{lemmaconvpgf} and Lemma \ref{lemmalhospital} below are
used in the proof of Lemma \ref{lemmaoffspringdist}. Lemma
\ref{lemmaconvpgf} provides an asymptotic statement for Laplace
transforms and generating functions respectively. Lemma
\ref{lemmalhospital} is a version of L'Hospital's rule, which
is stated for completeness.
\begin{lemma} \label{lemmaconvpgf}
   Let $\xi$ be a nonnegative real valued random variable with
   $m:=\me(\xi)<\infty$. Suppose that the distribution function
   $F$ of $\xi$ satisfies $1-F(x)\le Cx^{-\alpha}$ for all
   $x\ge 0$ for some $C<\infty$ and $\alpha>1$. Then, for every
   $\varepsilon\in [0,\min(\alpha-1,1))$,
   \begin{equation} \label{convlaplace}
      \lim_{\lambda\to 0}
      \frac{1-\varphi(\lambda)+\lambda m}{\lambda^{1+\varepsilon}}
      \ =\ 0,
   \end{equation}
   where $\varphi$ denotes the Laplace transform of $\xi$. If $\xi$
   takes only values in $\nz_0$, then, for the same range of values
   of $\varepsilon$ as above,
   \begin{equation} \label{convpgf}
      \lim_{s\to 1}\frac{(1-s)m-(1-f(s))}{(1-s)^{1+\varepsilon}}\ =\ 0,
   \end{equation}
   where $f$ denotes the pgf of $\xi$.
\end{lemma}
\begin{remark}
   The tail condition holds if $\me(\xi^\alpha)<\infty$, since, by
   Markov's inequality, $1-F(x)=\pr(\xi^{\alpha}>x^\alpha)
   \le x^{-\alpha}\me(\xi^\alpha)$.
\end{remark}
\begin{proof} (of Lemma \ref{lemmaconvpgf})
   Applying the well known formula $\me(g(\xi))=g(0)+\int_0^\infty
   g'(x)(1-F(x))\,{\rm d}x$, $g\in C^1([0,\infty))$, to the function
   $g(x):=e^{-\lambda x}-1+\lambda x$ yields
   \[
   \frac{\varphi(\lambda)-1+\lambda m}{\lambda^{1+\varepsilon}}
   \ =\ \frac{1}{\lambda^{\varepsilon}}\int_0^\infty
        (1-F(x))(1-e^{-\lambda x})\,{\rm d}x
   \ \le\ \int_0^1\frac{1-e^{-\lambda x}}{\lambda^\varepsilon}\,{\rm d}x
          + C\int_1^\infty\frac{1-e^{-\lambda x}}{(\lambda x)^\varepsilon
          x^{\alpha-\varepsilon}}\,{\rm d}x.
   \]
   Since $\varepsilon<1$,
   $\lim_{\lambda\to 0}(1-e^{-\lambda x})/\lambda^\varepsilon=0$ and the
   first integral converges to $0$ by the dominated convergence theorem. Since
   $(1-e^{-\lambda x})/(\lambda x)^\varepsilon$ is bounded uniformly in
   $\lambda$ and $x$, and $\alpha-\varepsilon>1$, the dominated convergence
   theorem is again applicable and the second integral converges to $0$.
   If $\xi$ takes only values in $\nz_0$ then (\ref{convpgf}) follows from
   (\ref{convlaplace}) via the substitution $\lambda:=-\log s$, $s\in(0,1)$,
   and the fact that $-\log s=(1-s)+O((1-s)^2)$ as $s\to 1$.\hfill$\Box$
\end{proof}
The situation in the following lemma is the one of L'Hospital's rule.
\begin{lemma}
   Let $c,x_0\in [-\infty,\infty]$.	Let $f,g:I\to\rz$ be continuously
   differentiable on an open interval $I$ containing $x_0$ or having $x_0$
   as a limit point if the limit is one-sided. Assume further that
   $g'(x)\ne 0$ for all $x\in I\setminus\{x_0\}$. Let $\alpha\in\rz
   \setminus\{1\}$. If either
   \begin{eqnarray*}
      \lim_{x\to x_0} g^{1-1/\alpha}(x)\ =\
      \lim_{x\to x_0} f(x)/g^{1/\alpha}(x)\ =\ 0
   \end{eqnarray*}
   or
   \begin{eqnarray*}
      \lim_{x\to x_0} g^{1-1/\alpha}(x)\ =\
      \lim_{x\to x_0} f(x)/g^{1/\alpha}(x)\ =\ \infty,
   \end{eqnarray*}
   and
   \begin{eqnarray}
      \lim_{x\to x_0} (\alpha f'(x)/g'(x)-f(x)/g(x))\ =\ c,
      \label{local_2}
   \end{eqnarray}
   then $\lim_{x\to x_0}f(x)/g(x)=c(\alpha-1)^{-1}$. If the limit
   (\ref{local_2}) does not exist it still holds that
   \begin{eqnarray*}
      \liminf_{x\to x_0}
      \bigg(\alpha\frac{f'(x)}{g'(x)}-\frac{f(x)}{g(x)}\bigg)
      & \le & \liminf_{x\to x_0} (\alpha-1)\frac{f(x)}{g(x)} \\
      & \le & \limsup_{x\to x_0} (\alpha-1)\frac{f(x)}{g(x)}
      \ \le \ \limsup_{x\to x_0}
              \bigg(\alpha\frac{f'(x)}{g'(x)}-\frac{f(x)}{g(x)}\bigg).
   \end{eqnarray*}
   \label{lemmalhospital}
\end{lemma}
\begin{proof}
   A straightforward computation shows that
   \begin{eqnarray*}
      \frac{1}{\alpha-1}\bigg(\alpha\frac{f'(x)}{g'(x)} - \frac{f(x)}{g(x)}\bigg)
      \ =\ \frac{f'(x)g^{-1/\alpha}(x)-(1/\alpha)
           g^{-1-1/\alpha}(x)g'(x)f(x)}{(1-1/\alpha)g^{-1/\alpha}g'(x)},
   \end{eqnarray*}
   where the numerator and the denominator are the derivatives of
   $f(x)g^{-1/\alpha}(x)$ and $g^{1-1/\alpha}(x)$ respectively.
   Thus the convergence of the left hand side to $c(\alpha-1)^{-1}\in
   [-\infty,\infty]$ implies $\lim_{x\to x_0}f(x)/g(x)=\lim_{x\to x_0}
   (f(x)/g^{1/\alpha}(x))/g^{1-1/\alpha}(x)=c(\alpha-1)^{-1}$.\hfill$\Box$
\end{proof}
The following two results are needed in the proof of Theorem \ref{main2}.
Lemma \ref{lemmauniformweakconv} contains a statement on uniform weak
convergence. The last result (Lemma \ref{dense}) provides a certain
dense subset of $\widehat{C}(\rz\times [0,\infty))$.
\begin{lemma}
   Let $(X_n)_{n\in\nz}$ be a sequence of real-valued random variables
   converging weakly to a real-valued random variable $X$. Then, for
   every bounded and continuous function $f:\rz\to\rz$ and $A,B>0$,
   \begin{eqnarray}
      \lim_{n\to\infty}
      \sup_{|a|\le A,|b|\le B}|\me(f(aX_n+b))-\me(f(aX+b))|\ =\ 0.
      \label{local}
   \end{eqnarray}
   If $f\in\widehat{C}(\rz)$, then (\ref{local}) even holds if the
   supremum is taken over $[-A,A]\times\rz$ instead of $[-A,A]\times[-B,B]$.
   \label{lemmauniformweakconv}
\end{lemma}
\begin{proof}
   For $n\in\nz$ define $g_n:\rz^2\to\rz$ via $g_n(a,b):=\me(f(aX_n+b))$,
   $a,b\in\rz$, and $g$ similarly with $X_n$ replaced by $X$. Fix $A,B>0$.
   Obtaining pointwise convergence of $g_n$ to $g$ from weak convergence,
   (\ref{local}) follows, in view of the Arzel\`a--Ascoli theorem, from
   the uniform equicontinuity of $\{g_n:n\in \nz\}$ on
   $K:=[-A,A]\times [-B,B]$, that is, for every $\varepsilon>0$ there
   exists $\delta>0$ such that $\max(|a-a'|,|b-b'|)<\delta$ implies
   $|g_n(a,b)-g_n(a',b')|<\varepsilon$ for all $n\in\nz$ and all
   $(a,b),(a',b')\in K$.

   Let $\varepsilon>0$. By Prohorov's theorem the family of distributions
   of the weakly convergent sequence $(X_n)_{n\in\nz}$ is tight. Thus, there
   exists $C\in (0,\infty)$ such that $\sup_{n\in\nz}\pr(|X_n|>C)<\varepsilon$
   and $\pr(|X|>C)<\varepsilon$. Using the uniform continuity of $f$ on $K$,
   choose $\delta>0$ such that $|x-y|<\delta(C+1)$ implies
   $|f(x)-f(y)|<\varepsilon$. Consequently,
   \begin{eqnarray*}
      |g_n(a,b)-g_n(a',b')|
      & = & |\me(f(aX_n+b))-\me(f(a'X_n+b'))|\\
      & \le & 2\varepsilon \|f\|
               + \me(1_{\{|X_n|\le C\}}|f(aX_n+b)-f(a'X_n+b')|)\\
      & \le & 2\varepsilon\|f\|+\varepsilon
   \end{eqnarray*}
   for $(a,b),(a',b')\in K$ with $\max(|a-a'|,|b-b'|)<\delta$, proving
   the first statement.

   If $f\in\widehat{C}(\rz)$ then there exists $L>0$ such that
   $|f(x)|<\varepsilon$ for all $|x|>L$. In particular (\ref{local}) holds
   for $B:=AC+L$. On the remaining area $[-A,A] \times(\rz\setminus[-B,B])$
   all the functions $g_n$ and $g$ are sufficiently small. More precisely,
   if $|a|\le A$ and $|b|>B$, then $|aX_n+b|>L$ on the event $\{|X_n|\le C\}$,
   hence
   \[
	|g_n(a,b)|
	\ =\ |\me(f(aX_n+b))|
	\ \le\ \varepsilon\|f\|+\me(1_{\{|X_n|\le C\}}|f(aX_n+b)|)
    \ \le\ \varepsilon \|f\|+\varepsilon
   \]
   for all $n\in\nz$, and similarly $|g(a,b)|\le\varepsilon\|f\|+\varepsilon$,
   which proves the additional statement.\hfill$\Box$
\end{proof}
\begin{lemma} \label{dense}
   Let $S:=\rz\times [0,\infty)$. The space of functions
   $f:S\to\rz$ of the form $f(x,y)=\sum_{i=1}^l g_i(x)h_i(y)$
   with $l\in\nz$, $g_1,\ldots,g_l\in\widehat{C}(\rz)$ and
   $h_1,\ldots,h_l\in\widehat{C}([0,\infty))$ is dense in
   $\widehat{C}(S)$.
\end{lemma}
\begin{proof}
   Two proofs are provided. The first proof is elementary and constructive.
   The second proof exploits the Stone--Weierstrass theorem for locally
   compact spaces.

   \vspace{2mm}

   {\bf Proof 1.} (elementary)
   Each $f\in\widehat{C}(S)$ can be transformed (with the additional
   definition $f(\pm\infty,y):=0$ for all $y\in [0,\infty)$ and
   $f(x,\infty):=0$ for all $x\in\rz$) into a map
   $\widetilde{f}\in C([0,1]^2)$ satisfying
   $\widetilde{f}(0,y)=\widetilde{f}(1,y)=\widetilde{f}(x,1)=0$
   for all $x,y\in [0,1]$ via
   \[
   \widetilde{f}(x,y)
   \ :=\ f\bigg(\frac{1}{1-x}-\frac{1}{x},\frac{y}{1-y}\bigg),
   \qquad x,y\in [0,1]^2.
   \]
   Thus, it suffices to verify that the space $D$ of functions
   $f:[0,1]^2\to\rz$ of the form $f(x,y)=\sum_{i=1}^l g_i(x)h_i(y)$
   with $l\in\nz$, $g_1,\ldots,g_l\in D_1:=\{g\in C([0,1]):g(0)=g(1)=0\}$
   and $h_1,\ldots,h_l\in D_2:=\{h\in C([0,1]):h(1)=0\}$ is dense in
   $\{f\in C([0,1]^2):f(0,y)=f(1,y)=f(x,1)=0$ for all $x,y\in [0,1]\}$.
   This is seen as follows. Let $m\in\nz$. For $i\in\{0,\ldots,m\}$
   define $x_i:=i/m$ and $g_i:[0,1]\to [0,1]$ via
   \[
   g_i(x)\ :=\ (1-m|x-x_i|)\,1_{\{|x-x_i|\le 1/m\}},\qquad x\in [0,1].
   \]
   Note that $g_0,\ldots,g_m$ form a partition of unity,
   i.e.~$\sum_{i=0}^m g_i(x)=1$ for all $x\in [0,1]$. Moreover,
   $g_1,\ldots,g_{m-1}\in D_1$. In the same manner define
   $y_j:=j/m$ and $h_j:[0,1]\to [0,1]$ via
   $h_j(y):=(1-m|y-y_j|)1_{\{|y-y_j|\le 1/m\}}$ for all $j\in\{0,\ldots,m\}$.
   Again, $h_0,\ldots,h_m$ form a partition of unity,
   i.e.~$\sum_{j=0}^m h_j(y)=1$ for all $y\in [0,1]$.
   Moreover, $h_0,\ldots,h_{m-1}\in D_2$. Now define $f_m:[0,1]^2\to\rz$ via
   \[
   f_m(x,y)
   \ :=\ \sum_{i,j=0}^m f(x_i,y_j) g_i(x)h_j(y)
   \ =\ \sum_{i=1}^{m-1}\sum_{j=0}^{m-1} f(x_i,y_j) g_i(x)h_j(y),
   \qquad x,y\in [0,1],
   \]
   where the last equality holds since $f(0,y)=f(1,y)=f(x,1)=0$ for all
   $x,y\in [0,1]$. From $g_1,\ldots,g_{m-1}\in D_1$ and
   $h_0,\ldots,h_{m-1}\in D_2$ it follows that $f_m\in D$. It remains
   to verify that $\lim_{m\to\infty}\|f_m-f\|=0$. Let $\varepsilon>0$.
   Since $f$ is uniformly continuous on $[0,1]^2$ there exists
   $\delta=\delta(\varepsilon)>0$ such that $|f(x',y')-f(x,y)|<\varepsilon$
   for all $x,y,x',y'\in [0,1]$ with $|x-x'|<\delta$ and $|y-y'|<\delta$.
   For all $x,y\in [0,1]$ it follows from $\sum_{i,j=0}^m g_i(x)h_j(y)=1$ that
   \[
   |f_m(x,y)-f(x,y)|
   =\bigg|\sum_{i,j=0}^m \big(f(x_i,y_j)-f(x,y)\big)g_i(x)h_j(y)\bigg|
   \le\sum_{i,j=0}^m |f(x_i,y_j)-f(x,y)| g_i(x)h_j(y).
   \]
   Now for each $(x,y)\in [0,1]^2$ there exist $i_0,j_0\in\{0,\ldots,m-1\}$
   (depending on $x$ and $y$) such that $x_{i_0}\le x\le x_{i_0+1}$ and
   $y_{j_0}\le y\le y_{j_0+1}$. Since $g_i(x)=0$ for all
   $i\in\{0,\ldots,m\}\setminus\{i_0,i_0+1\}$ and $h_j(y)=0$ for all
   $j\in\{0,\ldots,m\}\setminus\{j_0,j_0+1\}$ we conclude that
   \begin{eqnarray*}
      |f_m(x,y)-f(x,y)|
      & \le & |f(x_{i_0},y_{i_0})-f(x,y)| + |f(x_{i_0},y_{j_0+1})-f(x,y)|\\
      &     &  + |f(x_{i_0+1},y_{j_0})-f(x,y)| + |f(x_{i_0+1},y_{j_0+1})-f(x,y)|\\
      & \le & 4\varepsilon
   \end{eqnarray*}
   for all $m\in\nz$ with $m>1/\delta$. Thus,
   $\lim_{m\to\infty}\|f_m-f\|=0$.\hfill$\Box$

   \vspace{2mm}

   {\bf Proof 2.} (using the Stone--Weierstrass theorem)
   The space of functions $f:S\to\rz$ of the form
   $f(x,y)=\sum_{i=1}^l g_i(x)h_i(y)$ with $l\in\nz$,
   $g_1,\ldots,g_l\in\widehat{C}(\rz)$ and
   $h_1,\ldots,h_l\in\widehat{C}([0,\infty))$
   is a subalgebra of $\widehat{C}(S)$, which separates points and
   vanishes nowhere, whence is dense in $\widehat{C}(S)$ by the
   Stone--Weierstrass theorem (see, for example, \cite{debranges}). In
   \cite{debranges} the theorem is stated for complex-valued functions,
   but it remains true for real-valued functions. To see this, let
   $f\in\widehat{C}(S)\subseteq\widehat{C}(S,\cz)$ be arbitrary.
   By the theorem there exist $g_1,g_2,\ldots\in\widehat{C}(S,\cz)$
   such that $\lim_{n\to\infty}\|g_n-f\|=0$. Then $f_n:=\mathrm{Re}(g_n)\in
   \widehat{C}(S)$, $n\in\nz$, and $\|f_n -f\|\le\|g_n-f\|\to 0$ as
   $n\to\infty$.\hfill$\Box$
\end{proof}


\end{document}